\documentclass[12pt]{amsart}

\usepackage{verbatim}
\usepackage{framed}
\usepackage{amsfonts}
\usepackage{amsthm,amsmath}
\usepackage[integrals]{wasysym}
\usepackage{enumerate}
\usepackage{fancyhdr}
\usepackage{amssymb}
\usepackage{chemarrow}
 
\usepackage{mathtools}
\usepackage{hyperref}

\usepackage{appendix}

\usepackage{enumitem}

\usepackage{epsfig,epic,eepic,graphicx}

\usepackage[usenames,dvipsnames,svgnames]{xcolor}
\usepackage{tikz-cd}
\usepackage{mathtools}

\usetikzlibrary{arrows,arrows.meta}

\numberwithin{equation}{section}
\let\veps=\varepsilon

\def\R{{\mathbb R}}
\def\T{{\mathbb T}}

\def\and{\quad{\rm and}\quad}

\def\virgp{\raise 2pt\hbox{,}}
\def\cdotpv{\raise 2pt\hbox{;}}

\def\div{ \hbox{\rm div}\,  }

\numberwithin{equation}{section}
\newtheorem{thm}{Theorem}[section]
\newtheorem{lem}[thm]{Lemma}

\newtheorem{prop}[thm]{Proposition}
\newtheorem{cor}[thm]{Corollary}
\newtheorem{rmk}[thm]{Remark}

\newcommand{\ben}{\begin{eqnarray}}
\newcommand{\een}{\end{eqnarray}}

\newcommand{\beq}{\begin{eqnarray}}
\newcommand{\eeq}{\end{eqnarray}}

\newcommand{\beno}{\begin{eqnarray*}}
\newcommand{\eeno}{\end{eqnarray*}}

\definecolor{darkgreen}{rgb}{0,0.5,0}
\definecolor{darkblue}{rgb}{0,0,0.7}
\definecolor{darkred}{rgb}{0.9,0.1,0.1}
\definecolor{lightblue}{rgb}{0,0.51,1}

\begin{document}
\title[]{Stable self-similar singularity formation for infinite energy solutions of the incompressible porous medium equations}

\author[C. Collot]{Charles Collot}
\address[C. Collot]{\newline  Laboratoire de Math\'{e}matiques AGM, UMR CNRS 8088, Cergy Paris Universit\'{e}, 2 Avenue Adolphe Chauvin, 95302     Cergy-Pontoise Cedex, France}
\email{charles.collot@cyu.fr}

\author[C. Prange]{Christophe Prange}
\address[C. Prange]{\newline  Laboratoire de Math\'{e}matiques AGM, UMR CNRS 8088, Cergy Paris Universit\'{e}, 2 Avenue Adolphe Chauvin, 95302     Cergy-Pontoise Cedex, France}
\email{christophe.prange@cyu.fr}

 \author[J. Tan]{Jin Tan}
\address[J. Tan]{\newline Department of Mathematics,The Chinese University of Hong Kong, Shatin, Hong Kong, P.R. China}
\email{jtan@math.cuhk.edu.hk}

\date{\today}
 \subjclass[2010]{ }

\begin{abstract}

We consider a special class of infinite energy solutions to the inviscid  incompressible porous medium equations (IPM), introduced in Castro-C\'{o}rdoba-Gancedo-Orive \cite{CCGO}. The (IPM) equations then reduce to a one-dimensional nonlocal nonlinear equation, for which an explicit self-similar blow-up solution is found in \cite{CCGO}. We show the stability of this explicit blow-up solution by smooth enough perturbations, and identity a sharp regularity threshold below which it is unstable. The heart of our proof is a change of variables that transforms the study of finite time blow-up solutions, to the study of global-in-time solutions to the Proudman-Johnson equation, which is a reduced equation that appears for special classes of solutions of the two-dimensional Euler equations, and of the inviscid primitive equations (or hydrostatic Euler equations). Our main result is in fact the asymptotic stability with decay estimates for a family of steady states of this reduced equation. It thus also implies the corresponding stability of steady states for the associated classes of solutions to the Euler equations and inviscid primitive equations.
 

 \end{abstract}
\maketitle

\section{Introduction and main results}

\subsection{The incompressible porous medium equations (IPM)}

In this paper, we consider the 2D inviscid  incompressible porous medium (IPM) equation. The equation describes evolution of density carried by the flow of incompressible fluid that is determined via Darcy's law in the field of gravity:
\begin{equation}\label{IPM}
\partial_\tau \rho+ u\cdot\nabla \rho=0,\quad \div u=0,\quad  u+\nabla P=(0, \rho).
\end{equation}
Here $\tau\geq0$ is the time variable,  $\rho$ is the transported density and $u$ is the vector field describing the fluid motion, and $P$ is the pressure.  For more on flows in porous media, we refer to \cite{Be,N,FL}. Throughout this paper, we consider the spatial domain to be the infinite strip $(x, y)\in[-\pi, \pi]\times \R,$ or $\T\times \R$ that is periodic for $x\in [-\pi, \pi].$  In the first case,  due to the presence of boundaries,  $u$ also satisfies the boundary condition $u\cdot n=0$ for $x=\pm\pi.$

Whether smooth and localized solutions to \eqref{IPM} can blow up in finite time remains a challenging open problem, while numerical simulations indicate absence of singularities \cite{CGO}. In Castro-C\'{o}rdoba-Gancedo-Orive  \cite{CCGO}  a class of infinite energy solutions has been  considered on $\T\times \R,$  for which  the stream function $\psi$ of the velocity field has the form $\psi(\tau, x, y)=y f(\tau, x).$   By the fact that $\Delta \psi=  \partial_{x}\rho$,    the density necessarily has the expression $\rho(\tau, x, y)=yf(\tau, y)+ g(\tau, y).$  Substituting this expression in   \eqref{IPM},  one gets
\begin{equation*}
y (\partial_\tau \partial_{x} f+ f\partial_{x}^2 f-(\partial_{x} f)^2)=y\partial_{y} g \partial_{x}f - \partial_\tau g,
\end{equation*}
which indicates that  $g$ has to be $g(\tau, y)=y G(\tau)$
with $G(\tau)=\frac{1}{\pi}  \int_0^\tau \|\partial_{x} f(\bar\tau, \cdot)\|_{L^2(\mathbb{T})}^2\,d\bar\tau.$
 Looking for $f(\tau, x)=\int_{-\pi}^x b(\tau,  \bar x)\,d\bar x,$ we have that $b(\tau, x)$  satisfies the following  1D nonlocal nonlinear  transport equation:
\begin{equation}\label{eq-b}
\partial_\tau b+  \left(\int_{-\pi}^x b(\tau, \bar x) \,d\bar x\right) \partial_x b-b^2 +\frac{1}{\pi}  \int_{-\pi}^{\pi}b^2  \,dx=\frac{1}{\pi}\left( \int_{0}^\tau\int_{-\pi}^{\pi}b^2\,dx\,d\bar\tau \right) b,  
\end{equation}
supplemented with mean-free boundary condition 
\begin{equation}\label{mean-free-b}
\int_{-\pi}^\pi b(\tau, x)\,dx=0.
\end{equation}
Note that the solution constructed in this way will also give  a solution of  \eqref{IPM} on the domain  $[-\pi, \pi]\times \R$ with prescribed boundary condition. 
\medskip

 It is found in \cite{CCGO} that Equation \eqref{eq-b} has a class of explicit blow-up solutions of the form
\begin{equation} \label{id:explicit-blow-up}
b(\tau,x)=\frac{\mu}{\cos (\mu \tau)} \cos(x)
\end{equation}
for $\mu \neq 0$. The above solution blows up at time $\tau^*=\frac{\pi}{2\mu}$ with the following self-similar behavior to leading order
\begin{equation} \label{id:explicit-blow-up-2}
b(\tau,x)=\frac{1}{\tau^*-\tau} \cos(x)+O(\tau^*-\tau)
\end{equation}
as $\tau \to \tau^*$. Our motivation  of the present paper is the stability issue of this family of  blow-up solutions.

Let us also mention that very recently,    Kiselev and Sarsam proposed a 1D nonlocal equation in \cite{KS} to model the behavior of a boundary layer solution to the (IPM) equation posed on the 2D periodic half-plane,   moreover they established local well-posedness for the equation as well as finite time blow-up for a class of initial data.

\subsection{Connexion with other fluid models through the Proudman-Johnson equation} \label{other-fluids}

Given $b$ a solution to \eqref{eq-b} satisfying \eqref{mean-free-b}, we perform the following change of variables:
\begin{equation}\label{def-a}
a(t,x)=\frac{1}{\nu(\tau)}b(\tau,x), \qquad t=\int_0^\tau \nu(\bar \tau)d\bar \tau, \qquad \frac{d}{d\tau}\nu =\frac{\nu}{\pi}\int_0^{\tau} \int_{-\pi}^{\pi}b^2(\bar \tau,  x)\,dx\,d\bar\tau 
\end{equation}
with $\nu(0)=\nu_0>0$. Then the new function $a$ solves
\begin{equation}\label{eq-main}
\left\{ \begin{array}{l l }
\partial_t a+  \left(\int_{-\pi}^x a(t, \bar x) \,d\bar x\right) \partial_x a -a^2 +\frac{1}{\pi}  \int_{-\pi}^{\pi}a^
2\,dx=0, \\
a(0, x)=a_0(x), 
\end{array} \right.
\end{equation}
on $[0,T)\times [-\pi,\pi]$, supplemented by the mean-free condition
\begin{equation} \label{bc}
\int_{-\pi}^\pi a(t,x)dx=0.
\end{equation}
We notice that condition \eqref{bc} is preserved over time once it is satisfied initially. Equation \eqref{eq-main} has the following scaling symmetry. Given a solution $a$, then for any $\lambda>0$, the following scaling also gives a solution:
\begin{equation} \label{id:scaling-symmetry}
(t,x)\mapsto \frac{1}{\lambda}a\left(\frac{t}{\lambda},x \right).
\end{equation}

Equation \eqref{eq-main} is the so-called Proudman-Johnson equation \cite{PJ} that appears as a reduced model for special classes of solutions  both for the $2D$ incompressible primitive equations (also known as the hydrostatic Euler equations) and the $2D$ Euler equations. Local-well posedness and finite time blow-up (which is possible depending on the initial data) for the Proudman-Johnson equation was studied in \cite{CISY,CW,O,OZ} and references therein. In \cite{S,SS,SS2} it is found that smooth periodic initial data lead to global solutions, while the absence of periodicity or the lack of smoothness may lead to finite time blow-up. The authors develop an elegant Lagrangian approach based on the characteristics. Our goal here is to develop an Eulerian and quantitative approach to stability, in the continuation of \cite{CGIM,CIL} in order to be able to treat the $2D$ models later on.

Following \cite{PJ} one can derive \eqref{eq-main} as a reduced equation from a special class of solutions of the $2D$ Euler equations
$$
\left\{ \begin{array}{l l}
    U_t + (U\cdot \nabla )U=-\nabla P ,  \\
    \nabla \cdot U =0.
\end{array} \right. 
$$
One considers a velocity field $U(x,y)=(\partial_y \psi,-\partial_x \psi)$ on the strip $(x,y)\in [-\pi,\pi]\times \mathbb R$ with no-penetration boundary condition and with a stream function of the form $\psi(t,x,y)=ya(t,x)$. Then the function $a$ satisfies \eqref{eq-main}. This equation can also be obtained for a special class of axisymmetric solutions without swirl, see \cite{EW} for example.   We also mention that a similar construction of infinite energy solutions that blow up in finite time for the $3D$ Euler equations is established in \cite{PC}.

Equation \eqref{eq-main} can also be derived as a reduced equation for a special class of solutions of the $2D$ inviscid primitive  equations (IPE)
$$
\left\{ \begin{array}{l l}
    u_t + u\, u_X + w u_Z +p_X  = 0 ,  \\
    p_Z =0 , \\
    u_X + w_Z =0,  
\end{array} \right.
$$
set on $\big\{ -1\leq X \leq 1, \ -\pi \leq Z\leq \pi \big\}$, with boundary conditions $w(t,X,-\pi) = w(t,X,\pi) = 0$ and $\int_{-\pi}^\pi u(t,-1,Z) dZ = \int_{-\pi}^\pi u(t,1,Z) dZ = 0$, what was firstly done in \cite{CINT}. For $u(X)=-u(-X)$ an odd in $X$ solution, let $ a(t,Z)=-\partial_X u(t,0,Z)$ denote the trace of their horizontal derivative on the central line. Then $a(t,x)$ solves \eqref{eq-main}. In particular, for the particular class of solutions of the form $u(t,X,Z)=-Xa(t,Z)$, the (IPE) evolution is equivalent to the reduced equation \eqref{eq-main}. We mention that the hydrostatic Euler equations can be derived as an asymptotic model from the Euler equations \cite{Br,G}.

\medskip

Interestingly, the blow-up solutions \eqref{id:explicit-blow-up} for the incompressible porous medium equations \eqref{eq-b} corresponds to $  \tau=\frac{2}{\mu}\arctan\left(1-\frac{2}{e^{\mu t}+1}\right)$ and $\nu= \frac{1}{\cos(\mu \tau)}$ and
   \begin{equation}\label{formula-t}
a(t,x)= \mu \cos x.
  \end{equation}
Thus, the family of finite-time blow-up solutions \eqref{id:explicit-blow-up} to \eqref{eq-b} corresponds to the family of stationary solutions \eqref{formula-t} to \eqref{eq-main}. It turns out  that  when $\mu>0,$ the  stability  of these blow-up solutions of \eqref{eq-b} is equivalent to  the  global-in-time stability  of the  steady states $\{\mu\cos(x)\}_{\mu\in\R\backslash\{0\}}$  for  Equation \eqref{eq-main},  via the change of variables \eqref{def-a}.    As a consequence of this observation, the strategy of the present article is to consider the stability issue of the stationary states \eqref{formula-t} for \eqref{eq-main}. 

Equation \eqref{eq-main} admits more generally as steady states the functions $\mu \cos(kx)$ and $\mu \sin(\frac{2k+1}{2}x)$ for $k\in \mathbb N$ and $\mu\in \mathbb R$. We show in Proposition \ref{Le-SS} that these are the only stationary states. We believe that stationary states other then $\mu \cos (x)$ for $\mu>0$ are unstable.

\medskip

The precise understanding of the large-time dynamics or blow-up dynamics for the inviscid primitive equations remains largely open. Blow-up was first proved to occur in \cite{CINT,Wo}. In \cite{CINT},  a family of backward self-similar solutions has been constructed to Equation \eqref{eq-main}.  However,  its stability is not known given the finite smoothness of these profiles.   On the other hand,    stable blow-up solutions around smooth and nonsmooth nearly self-similar blowup profiles have been constructed in  \cite{CIL}. Steady states are studied in \cite{LWX}.

The hydrostatic Euler equations have another instability apart from blow-up: they may be ill-posed in Sobolev spaces \cite{R,HN}. The equation is thus only well-posed in the general case for analytic data \cite{KTVZ,KTVZ2}, or for solutions with a monotonicity condition \cite{MW,KMVW}. Non-uniqueness of weak solutions was studied recently \cite{CM}, and in connexion with the Onsager conjecture \cite{BMT,BMT2}. For more result on this equation and other models of geophysical flows we refer to the survey \cite{LT}.

\medskip

Finally, let us mention that Equation \eqref{eq-main} shares similarities with other $1D$ reduced models. First, for the viscous Prandtl equation on the axis \cite{EE}: 
\begin{equation}
\partial_t v -\partial_y^2 v +\left( \int_0^y v(t,  \bar y)\,d\bar y\right)\,\partial_y v -v^2=0,\quad y>0
\end{equation}
with  boundary condition  $v(t,   0)=0.$  
In \cite{CGIM},   a blow-up dynamics is found,  for which the viscosity is negligible,  where the singularity forms on a large spatial scale $y\sim (T-t)^{-\frac{1}{2}},$  see also  \cite{CGM, CGM2, KVW}. Second, for the De Gregorio equation
\begin{equation}\label{DGE}
\partial_t\omega +u\partial_\theta \omega=\omega\partial_\theta u,\quad \partial_\theta u=\mathcal{H}\omega,\quad \int_{\mathbb{S}^1} u\,d\theta=0,
\end{equation}
where $\mathcal{H}$ is the Hilbert transform and $\mathbb{S}^1$ is the unit circle, Jia-Stewart-Sverak \cite{JSS} established stability of steady state $\sin(\theta)$ for \eqref{DGE}.   Their proof is based on a careful spectrum analysis of the linearized operator $L(\omega)=\sin(\theta)\partial_\theta(\omega+ u)-\cos(\theta)(\omega+u)$ (which has fine commutator structure)  and by energy estimates.

\subsection{Stable self-similar blow-up for infinite energy solutions of (IPM)}

Our first main result is the stability of the family of blow-up solutions \eqref{id:explicit-blow-up} for $\mu>0$. Small and regular enough initial perturbations will give solutions that still blow up in finite time, and that have the same self-similar behavior \eqref{id:explicit-blow-up-2} near the blow-up time. 
  
\begin{thm}\label{Th-IPM}

There exists $\delta>0$ such that for any $\mu>0$ the following hold. Let $b_0\in  C^{3}(\Omega)$ satisfy the mean free condition $\int_{-\pi}^\pi b_0=0$ and such that
\begin{equation}\label{Th-ic-IPM}
\|b_0(\cdot) -\mu \cos ( \cdot)\|_{C^3([-\pi,\pi])} \leq \delta \mu,
\end{equation}
then the unique $C^3$ solution $b(t, x)$ of the equation \eqref{eq-main} with $b(0, x)=b_0(x)$ blows up in finite time $\tau^*, $ for some $\tau^*>0$.   Moreover, it can be decomposed as
\begin{equation}\label{b-behavior}
b(\tau,x)=\frac{1}{\tau^*-\tau}\cos(x)+\tilde b(\tau,x),
\end{equation}
where
\begin{equation}\label{b-behavior-remainder}
\| \tilde b(\tau)\|_{C^1([-\pi,\pi])}\lesssim (\tau^*-\tau)^{-3/4}.
\end{equation}

\end{thm}

We expect that for $\mu<0$, the solution \eqref{id:explicit-blow-up} is unstable, see Remark \ref{rk:ipe}. Since Equation \eqref{eq-b} does not have sign reversal symmetry, such solutions are actually very different from the solutions for $\mu>0$. The key difference is that for $\mu>0$ the function $\mu \cos(x)$ attains its maximum in the interior of $(-\pi,\pi)$ while for $\mu<0$ it is attained at the boundary $\{-\pi,\pi\}$. We make further remarks on Theorem \ref{Th-IPM} in Remark \ref{rk:ipe}.

The stability proved in Theorem \ref{Th-IPM} only holds for smooth enough perturbations. In the next theorem, we show instability for certain small perturbations in $C^{2-\epsilon}$ for any $0<\epsilon\leq 2$.

\begin{thm} \label{Th-main-4} 
For any $\epsilon,  \delta>0$ and $\mu>0$ there exists an initial data $b_0$ which satisfies the mean free condition $\int_{-\pi}^\pi b_0=0$ and
\begin{equation}\label{Th-ipm-4}
\|b_0 - {\mu}\cos ( x)\|_{C^{2-\epsilon}([-\pi,\pi])}\leq \delta \mu,
\end{equation}
such that the problem \eqref{eq-main} cannot have a solution blowing up at some time $\tau^*>0$ that can be decomposed as
\begin{equation}\label{TH4-behavior}
b(\tau,x)=\frac{1}{\tau^*-\tau}\cos(x)+\tilde b(\tau,x), \qquad 
\end{equation}
with
\begin{equation}\label{TH4-behavior-2}
(\tau^*-\tau)\| \tilde b(\tau)\|_{L^\infty([-\pi,\pi])} \to 0 \mbox{ as }\tau\to \tau^*.
\end{equation}
\end{thm}

 {Notice that the solution $b$ constructed in Theorem \ref{Th-main-4} may blow-up, but not in the nearly self-similar form given by \eqref{b-behavior}-\eqref{b-behavior-remainder}.}

\subsection{Stability for the family of stationary states of the Proudman-Johnson equation}

The result at the heart of the present article is the asymptotic stability by smooth enough perturbations of the family of stationary states $\{\mu \cos(x)\}_{\mu>0}$ to the Proudman-Johnson equation \eqref{eq-main}. This implies asymptotic stability results for the special classes of solutions to the inviscid primitive equations (or hydrostatic Euler equations), and Euler equations, as discussed in Subsection \ref{other-fluids}.

\begin{thm}\label{Th-main}

There exists $\delta>0$ such that for any $\mu>0$ the following hold. Let $a_0\in C^{3}(\Omega)$ satisfy the mean free condition $\int_{-\pi}^\pi a_0=0$ and
\begin{equation}\label{Th-icmain}
\|a_0(\cdot) -\mu \cos ( \cdot)\|_{C^3([-\pi,\pi])}= {\sigma} \leq \delta \mu.
\end{equation}
Then the equation \eqref{eq-main} has a unique global-in-time $C^3$ solution $a(t, x)$ with $a(0, x)=a_0(x).$ Moreover,  it satisfies
\begin{equation}\label{a-behavior}
\|a(t, \cdot)-\mu_* \cos(\cdot)\|_{C^1(\Omega)}\lesssim  {\sigma} \, e^{-\mu t/2},
\end{equation}
for all $t\geq 0$, where $\mu_*=-\partial_x^2 a_0(x^*_0)=\mu+O(\sigma)$ with $x_0^*$ being the point at which $a_0$ attains its maximum.

\end{thm}

Let us make a few remarks on Theorem \ref{Th-main}.

\begin{rmk} \label{rk:ipe}\ \\
\emph{- The point $x_0^*$ is well-defined.} For $\delta>0$ small enough, we remark that $x^*_0$ is well defined (see Lemma \ref{lem:characterization-x0*}) by standard differential calculus arguments.

\smallskip

\noindent \emph{- Determination of the asymptotic state by a conservation law.} The fact that the asymptotic stationary state among the full family of stationary states $\{\tilde \mu \cos ( x) \}_{\tilde \mu>0}$ can be determined from the initial datum, via the identity $\mu_*=-\partial_x^2 a_0(x^*_0)$, is remarkable. It is due to monotonicity formulas and conserved quantities along characteristics. We refer to section \ref{sec:monotonicity-characteristics} for more details.

\smallskip

\noindent \emph{- Optimal regularity range for the initial perturbation for stability.} The $C^3$ regularity is the simplest regularity to prove Theorem \ref{Th-main}. In fact, our arguments could yield stability, in the sense that the conclusions of the Theorem still hold true, for initial small perturbation in $C^{2+\epsilon}$ for any $\epsilon>0$. It would only be a technical extension.

\smallskip

\noindent \emph{- Different regularities for stability and for measuring convergence}. We just discussed that the initial perturbation needs to be $C^3$ (or $C^{2+\epsilon}$) in order for the solution to converge to a stationary state in the sense of \eqref{a-behavior}. The convergence in \eqref{a-behavior} is stated in $C^1$, and in fact our arguments could yield convergence in $C^{2-\epsilon}$ for any $0<\epsilon< 1$; it would only be a technical extension. However, it is to be noted that we believe this convergence \emph{does not hold} in $C^2$. This is because the linearized dynamics near $-\pi$ (and also near $\pi$) has a potential term that damps the perturbation but a transport term that compresses the characteristics. Namely, to leading order the dynamics for the perturbation $\tilde a$ near $-\pi$ is $\partial_t \tilde a = -2\tilde a +(x+\pi)\partial_x \tilde a +l.o.t$. So we expect that near $-\pi$ the perturbation $\tilde a$ has the form $\tilde a = e^{-2t}f(e^t(x+\pi))+l.o.t.$. Clearly, all derivatives of order $<2$ decay to $0$ while derivatives of order $>2$ grow and the second order derivative does not decay.

\smallskip

\noindent \emph{- On the stability of $\mu \cos(x)$ for $\mu<0$.} We remark that $\mu \cos(x)$ is also a stationary solution to \eqref{eq-main} for $\mu<0$. In the class of $C^3([-\pi,\pi])$ solutions, we expect such stationary states to be unstable, because their maximum is attained at the boundary $\{-\pi,\pi\}$. Perturbed initial data could have their maximum attained in the interior $(-\pi,\pi)$ in which case we believe the characteristics will push this maximum towards the center. The stationary states for $\mu>0$ and $\mu<0$ are thus completely different when it comes to their stability.

\end{rmk}

In fact, the asymptotic stability proved in Theorem \ref{Th-main} only holds for smooth enough perturbations. In the next theorem, we show instability for certain small perturbations in $C^{2-\epsilon}$ for any $0<\epsilon\leq 2$.  Note that the instability of $\cos(x)$ will imply that  of the family of stationary solutions $\{\mu\cos(x)\}_{\mu>0},$ thanks to the symmetry scaling \eqref{id:scaling-symmetry}.

\begin{thm} \label{Th-main-2}
For any $\epsilon,  \sigma>0$ there exists an initial data $a_0$ which satisfies the mean free condition $\int_{-\pi}^\pi a_0=0$ and
\begin{equation}\label{Th-ic}
\|a_0 -\cos ( x)\|_{C^{2-\epsilon}(\Omega)}\leq \sigma,
\end{equation}
such that for all $\mu\in\R\setminus\{0\}$ the problem  \eqref{eq-main} cannot have a global-in-time solution that satisfies
\begin{equation}\label{TH2-behavior}
\|a(t, \cdot)-\mu \cos(\cdot)\|_{L^\infty([-\pi,\pi])}\to 0,\qquad t\to\infty.
\end{equation}
\end{thm}

\begin{rmk}

It is clear from the proof that the solution we construct exhibits exponential instability in $C^{2-\epsilon}$ for a first period of times. The result of Theorem \ref{Th-main-2} is actually stronger than a short time instability  as it holds for large times: the solution can never converge back to any stationary state in the family $\{\mu \cos(x)\}_{\mu\neq 0}$.

\end{rmk}

\bigbreak

\subsection{Organization of the article and sketch of proof}
Here,  we give a sketch of proof of the nonlinear stability  of  the steady state $\cos( x)$ for Equation \eqref{eq-main}. We recall that Equation \eqref{eq-main} admits the scaling symmetry \eqref{id:scaling-symmetry}. Therefore, the stability of $\cos(x)$ will imply that of $\mu \cos(x)$ for $\mu>0$ by a scaling argument. However,  due to the presence of nonlocal term $\int_{-\pi}^x a(t,  \bar x)\,d\bar x$,  Equation \eqref{eq-main} is not translation invariant.

\medskip

 Define the perturbation $\varepsilon=a-\cos ( x),$ it is easy to find that
\begin{equation}\label{eq-vareps}
\partial_t \veps+  \left(\int_{-\pi}^x \veps(t, \tilde x) \,d\tilde x\right) \partial_x \veps  -\veps^2 +\frac{1}{\pi}  \int_{-\pi}^{\pi}\veps^2\,dx=\mathcal{L}_0(\veps),\quad x\in [-\pi,\pi],
\end{equation}
where the associated linear operator $\mathcal{L}_0$ takes the form
\begin{equation*} 
\mathcal{L}_0(\veps)= 2\cos( x) \,\veps -\sin ( x)\,\partial_x \veps +\sin( x) \int_{-\pi}^x \veps(t, \bar x) \,d\bar x - \frac{2}{\pi}  \int_{-\pi}^{\pi}\cos ( x)\, \veps(t,   x)\,dx.
\end{equation*}
 
 In Section \ref{S2},  we first classify all stationary solutions of \eqref{eq-main}, which tells all the possible stationary solutions nearby $\cos(x).$
 Then,  we  present some monotonicity formulas and conserved quantities along characteristics,  that will later help to select the infinite-time attractor of the global solution, and to perform a change of variables for the nonlinear stability analysis.
 
 \medskip
 
  In Section \ref{sec3},   we focus on linear analysis of the problem.   We provide a key Proposition \ref{Prop-quasi} on \emph{weighted linear  damping} for a linearized problem with {\em well-prepared initial data},  on a moving spatial interval perturbated around $[-\pi, \pi]$.     This result is achieved at first for the  {\em reduced linear operator} 
\begin{equation}\label{eq-linearL0}
 L_0 (\veps)= 2 \cos ( x) \,\veps-\sin ( x)\, \partial_x \veps
\end{equation}
with general admissible weights via comparison principle. For the original operator $\mathcal{L}_0$, in addition to the terms in $L_0$, it is needed to analyze the two nonlocal terms.    By providing a well-designed admissible weight,  we are able to handle the first nonlocal term  $\sin(x)\int_0^x \veps(\bar x)\,d\bar x$ involved in the  linearized problem. Combining with the comparison principle, we finally obtain Proposition \ref{Prop-quasi}. The second nonlocal term $\frac{2}{\pi}  \int_{-\pi}^{\pi}\cos ( x)\, \veps(t,   x)\,dx$ is handled via modulation as explained below.
 
 \medskip
 
Section \ref{S4} is devoted to reformulating the problem \eqref{eq-vareps}  as a similar problem to that which was studied at the linear level  in Section \ref{sec3}.   It will be divided into two steps.  In the first step,  we introduce suitable  rescaling  and   space-translation $x^*(s)$ to  modulate the parameters of the family of functions $\{\mu \cos(x-x^*)\}_{\mu,x^*}$ and 
transform the full linear operator $\mathcal{L}$ into a sum  of $L_0$ with two nonlocal terms
$$  \sin(y)\int_0^y \eta(\bar y)\,dy\and -\frac{2}{\pi}\int_{-\pi}^{\pi}\cos(y)\eta(y)\,dy$$
where $\eta$ denotes the new perturbation. In the second step,  we decompose $\eta$ into its projection on certain eigenmodes of $\mathcal{L}$, plus a remainder part. The modulation parameters for the projection are chosen so as to remove the second nonlocal term $-\frac{2}{\pi}\int_{-\pi}^{\pi}\cos(y)\eta(y)\,dy$. After these operations,   the  original problem \eqref{eq-vareps} will finally be recast to problem \eqref{eq-xi}--\eqref{definition-xi0}.
Section \ref{S5} is concerned with the nonlinear analysis i.e.  the proof of Theorem \ref{Th-main}. It mainly consist in closing the desired estimates by applying  Proposition \ref{Prop-quasi} and estimating the projection parameters. Interestingly, one eigenmode is instable, but its decay is ensured for the full nonlinear problem thanks to the mean-free condition \eqref{bc}, highlighting its stabilizing effect.

\medskip

 The proof of Theorem \ref{Th-main-2} in Section \ref{sec:instability} is based on the use of the characteristics for a well-designed perturbation producing a cusp-like maximum, see Lemma \ref{lem:technical-instability} for more details. The proofs of Theorems \ref{Th-IPM} and \ref{Th-main-4} in Section \ref{sec:IPM-results} are obtained via the change of variables \eqref{def-a}.


\section{Steady states, conserved quantities and monotonicity formulas}\label{S2}

\subsection{Classification of steady states}

We provide in this subsection a classification of the steady states of \eqref{eq-main}. It is remarkable that they are given by simple expressions. Steady states are important since they not only provide explicit solutions, but also natural candidates for describing the large-time behavior of general solutions.

Steady states of the system \eqref{eq-main} are functions $a$ solving the stationary equation

\begin{equation} \label{eq-main-stationary}
 \left(\int_{-\pi}^x a(t, \bar x) \,d\bar x\right) \partial_x a -a^2 +\frac{1}{\pi}  \int_{-\pi}^{\pi}a^
2\,dx=0 .
\end{equation}

It is easy to check that $\{\mu\cos(kx)\}_{\mu\in\R, \,k\in\mathbb N}$ and $\{\mu\sin(\frac{2k+1}{2}x)\}_{\mu\in\R, \,k\in\mathbb N}$ are two classes of  stationary solutions of \eqref{eq-main}.   In fact,  under mild assumptions,  these are the only stationary solutions to the equation with mean-free boundary condition.

\begin{prop}[Classification of steady states]\label{Le-SS}
A function $a \in C^2([-\pi, \pi])$ satisfies the mean free condition $\int_{-\pi}^\pi a=0$ and solves the stationary equation \eqref{eq-main-stationary} if and only if it is either of the form $a=\mu\cos(kx)$ for some $\mu\in\R$ and $k\in\mathbb N$ or of the form $a=\mu\sin(\frac{2k+1}{2}x)$ for some $\mu\in\R$ and $k\in\mathbb N.$
\end{prop}

\begin{rmk}
Among all stationary states, the ones of the form $\mu \cos(x)$ for $\mu>0$ are the only ones attaining their maximum at a single point, inside $(-\pi,\pi)$. We believe that this is the key property that makes them stable and that other stationary states should be unstable.
\end{rmk}

\begin{proof}
We assume $a\neq 0$. Define $A(x)=\int_{-\pi}^x a(\bar x)\,d\bar x$,  then from \eqref{eq-main-stationary}, $A$ satisfies
\begin{equation} \label{mamiche}
A\, A'''= A'\,A'' .
\end{equation}

\noindent \textbf{Step 1}. \emph{Study on a non-degenerate interval}. Let $I=(x_1,x_2)\subset [-\pi,\pi]$ be an open interval on which $A,A''\neq 0$ and such that $A(x_1)=A(x_2)=0$. We claim that
\begin{equation} \label{reveil}
A= \mu_I \cos(\kappa_I x+ \varphi_I) \qquad \mbox{for }x\in I,
\end{equation}
for some constants $\kappa_I>0$, $\mu_I\neq 0$ and $\varphi_I\in \mathbb R$.

We now prove this claim. Dividing both sides of \eqref{mamiche} by $AA''$ we get $A'''/A''=A'/A$ on $I$. Hence $\frac{d}{dx} \log |A''|=\frac{d}{dx}\log |A|$ so that
\begin{equation} \label{sylvette}
A''=m_I A
\end{equation}
for some integration constant $m_I$.

We now claim that $m_I<0$. Suppose by contradiction that $m_I>0$ (notice $m_I$ cannot be $0$ as $A''\neq 0$ on $I$). Since $A''\neq 0$ on $I$, there must exist a point $x_*$ at which $A(x_*)\neq 0$ and $A'(x_*)\neq 0$. We consider first the case where $A(x_*)>0$ and $A'(x_*)>0$ at this point. Then the ODE \eqref{sylvette} implies that $A$ is increasing on $[x_*,x_2]$, so $A(x_2)\geq A(x^*)>0$, contradicting the fact that $A(x_2)=0$.  Second, we consider the case where $A(x_*)>0$ and $A'(x_*)<0$. Then the ODE \eqref{sylvette} implies that $A$ is decreasing on $[x_1,x_*]$, so that $A(x_1)\geq A(x_*)>0$, contradicting the fact that $A(x_1)=0$. Third we consider the cases where $A(x_*)<0$ and $A'(x_*)<0$, or  $A(x_*)<0$ and $A'(x_*)>0$, and we notice that by applying the same results as for the first and second cases to $-A$, we also get a contradiction. Hence all cases yield a contradiction, so that $m_I<0$.

We introduce $\kappa_I=\sqrt{-m_I}$, and all solutions to \eqref{sylvette} are indeed of the form \eqref{reveil}.

\medskip

\noindent \textbf{Step 2}. \emph{Matching of two non-degenerate intervals}. In this step, we claim that if $A$ is a $C^3$ function  {(remember that $a\in C^2$ and $A'=a$)} near $x_0$ at which $A(x_0)=0$, such that on the left of $x_0$ we have $A(x)= \mu \cos(\kappa x+\varphi)$ and on the right of $x_0$ we have $A(x)= \mu' \cos(\kappa' x+\varphi')$ for some $\mu,\mu'\neq 0$, $\kappa,\kappa'>0$ and $\varphi,\varphi'\in \mathbb R$, then $ \mu \cos(\kappa x+\varphi)= \mu' \cos(\kappa' x+\varphi')$ for all $x\in \mathbb R$, i.e. $(\mu,\kappa,\varphi)=(\mu',\kappa',\varphi')$ (possibly up to adding to $\varphi$ a multiple of $2\pi$, or to adding $\pi$ to $\varphi$ and changing $\mu$ to $-\mu$).

We now prove this claim. Since both formulas on the left and right of $x_0$ vanish at $x_0$, we necessarily have $A(x)=\mu \cos(\kappa (x-x_0)+\frac \pi 2)$ and $A(x)=\mu' \cos(\kappa' (x-x_0)+\frac \pi 2)$ on the left and right of $x_0$ respectively, up to changing the sign of $\mu$ and $\mu'$ without loss of generality. Then matching the first order and third order derivatives at $x_0$ implies $\kappa=\kappa'$ and $\mu=\mu'$. Hence the claim.

\medskip

\noindent \textbf{Step 3}. \emph{End of the proof}. We have that $a=A'$ is not identically zero, with $A(-\pi)=0=A(\pi)$. Hence $A$ and thus $A''$ is not identically zero. Therefore, there exists at least one point $x_0\in (-\pi,\pi)$ at which $A''(x_0)\neq 0$ and $A(x_0)\neq 0$. Let $(x_{-1},x_{1})$ be the largest interval in $[-\pi,\pi]$ containing $x_0$ on which $A\neq 0$ and $A''\neq 0$. Then by \eqref{sylvette} we see that $A(x_{-1})=0$, since otherwise $x_{-1}>-\pi$ with $A(x_{-1})\neq 0$ and $A''(x_{-1})\neq 0$ so one could find a larger interval on which $A,A''\neq 0$ which would contradict the definition of $(x_{-1},x_1)$. Similarly, $A(x_1)=0$.

By Step 1 we have that $A$ is of the form $A= \mu \cos(\kappa x+ \varphi)$ on $(x_{-1},x_1)$ with $\mu\neq 0$ and $\kappa>0$ and $\varphi\in \mathbb R$. Suppose $x_1\neq \pi$. Then we must have $A(x_1)=A''(x_1)=0$. Using $A= \mu \cos(\kappa x+ \varphi)$ on $(x_{-1},x_1)$, we get that $A'(x_1),A'''(x_1)\neq 0$. Hence there exists an interval of the form $(x_1,x_2)$ on which $A,A''\neq 0$ and which is the largest interval of the form $(x_1,y)$ for $y\in (x_1,\pi]$ on which this holds. By Step 1 we have that $A$ is of the form $A= \mu' \cos(\kappa' x+ \varphi')$ on $(x_{1},x_2)$. By Step 2 we must have (without loss of generality) $\mu'=\mu$ and $\kappa'=\kappa$ and $\varphi'=\varphi$. Hence $A$ is of the form $\mu \cos(\kappa x+ \varphi)$ on $(x_{-1},x_2)$.

One then repeats this argument and find points $x_3$, $x_4$ and so on until one reaches $x_k=\pi$ and find that $A$ is of the form $\mu \cos(\kappa x+ \varphi)$ on $(x_{-1},\pi)$. Note that there are finitely many steps since $x_{i}-x_{i-1}=\frac{\pi}{\kappa}$ from the formula for $A$ and the facts that $A(x_i)=0$ for $i=1, 2, \cdots$.  One then proceeds the same way on the left and find points $x_{-2}$, $x_{-3}$, and so on until one reaches $x_{-k'}=-\pi$ and find eventually that $A=\mu \cos(\kappa x+ \varphi)$ on $(-\pi,\pi)$. The boundary conditions $A(-\pi)=A(\pi)=0$ then imply that $\kappa \pi +\varphi =\frac{2n_1+1}{2}\pi$ and $-\kappa \pi +\varphi =\frac{2n_2+1}{2}\pi$ for some $n_1,n_2\in \mathbb N$. Introducing $n=n_1-n_2$, this gives $\varphi=(\frac{n+1}{2}+n_2)\pi$ and $\kappa=\frac{n}{2}$. Considering the even case $n=2k$ or the odd case $n=2k+1$ we have
$$
A(x)=\mu \cos \Big(k x+\frac \pi 2\Big) \quad \mbox{or} \quad A(x)=\mu \cos \Big(\frac{2k+1}{2} x\Big)
$$
up to possibly changing $\mu$ to $-\mu$. This gives
$$
a(x)=k\mu \sin \Big(kx+\frac \pi 2\Big)=k\mu \cos(kx) \quad \mbox{or} \quad a(x)=\frac{2k+1}{2}\mu \sin \Big(\frac{2k+1}{2} x\Big)
$$
which ends the proof of the Proposition up to redefining $\mu$.
\end{proof}

\subsection{Monotonicity formulas and conserved quantities along characteristics} \label{sec:monotonicity-characteristics}

In this subsection we find some monotonicity formulas for a solution $a$ to \eqref{eq-main} and its derivatives along characteristics. We say that a time-dependent point $x^*$ is a characteristics starting from $x^*_0$ if 
$$
\frac{d}{dt} x^*(t)=\int_{-\pi}^{x^*(t)} a(t,x)dx \qquad \mbox{and}\qquad x^*(0)=x^*_0.
$$
Our first result is that maxima of solutions are transported along characteristics.

\begin{lem} \label{lem:propagation-maximum}

Assume that $a\in C^1([0,T]\times[-\pi,\pi])$ solves \eqref{eq-main}. Assume that $x\mapsto a(0,x)$ attains its maximum at $x_0^*$. Then for all $t\in [0,T]$, $x\mapsto a(t,x)$ attains its maximum at the characteristics $x^*$ starting from $x^*_0$.

\end{lem}

\begin{proof}
Let $a^*(t)=a(t,x^*(t))$, which satisfies $\frac{d}{dt} a^*(t)=(a^*(t))^2-\frac{1}{\pi}\int_{-\pi}^\pi a^2(t,x)dx$ and $a^*(0)=\max_{[-\pi,\pi]}a(0,\cdot)$. Then the function $\tilde a=a^*(t)-a(t,x)$ solves the linear transport equation
\begin{equation} \label{propagation-maximum-equation-difference}
\partial_t \tilde a+b\partial_x \tilde a+c\tilde a =0
\end{equation}
where $b(t,x)=\int_{-\pi}^x a(t,\bar x)d\bar x$ and $c(t,x)=-(a(t,x)+a^*(t))$. Since initially $\tilde a(0,x)\geq 0$ for all $x\in[-\pi,\pi]$, this remains so, i.e. $ a^*(t)\geq a(t,x)$ for all $x\in[-\pi,\pi]$ and $t\in [0,T]$, due to the comparison principle with the $0$ solution for Equation \eqref{propagation-maximum-equation-difference}.
\end{proof}

Our second result is that the zero set of the derivative $\partial_x a$ is transported along characteristics.

\begin{lem} \label{lem:propagation-zero-set-derivatives}

Assume that $a\in C^1([0,T]\times[-\pi,\pi])$ solves \eqref{eq-main}. Assume that $\partial_x a_0(x^*_0)=0$ at some point $x_0^*$. Then for all $t\in [0,T]$, we have $\partial_x a(t,x^*(t))=0$ at the characteristics $x^*$ starting from $x^*_0$.

\end{lem}

\begin{proof}

Differentiating \eqref{eq-main}, we obtain that $\partial_x a$ solves the transport equation
$$
\partial_t (\partial_x a)+  \left(\int_{-\pi}^x a(t, \bar x) \,d\bar x\right) \partial_x (\partial_x a) - a (\partial_x a ) =0, \quad x\in\Omega.
$$
Solving this equation along characteristics, we obtain that if $\partial_x a(x^*_0)=0$ then $\partial_x a(x^*(t))=0$ for all times. This is the desired result.
\end{proof}

Our third result is that the values of the second order derivative $\partial_x^2 a$ are conserved along characteristics in the zero set of $\partial_x a$.

\begin{lem} \label{lem:conservation-second-derivative-zero-set-derivatives}

Assume that $a\in C^2([0,T]\times[-\pi,\pi])$ solves \eqref{eq-main}. Assume that $\partial_x a_0(x^*_0)=0$ at some point $x_0^*$. Then for all $t\in [0,T]$, we have $\partial_x^2 a(t,x^*(t))=\partial_x^2 a_0(x^*_0)$ at the characteristics $x^*$ starting from $x^*_0$.

\end{lem}

\begin{proof}

Differentiating twice \eqref{eq-main}, we obtain that $\partial_x^2 a$ solves the transport equation
$$
\partial_t (\partial_x^2 a)+  \left(\int_{-\pi}^x a(t, \bar x) \,d\bar x\right) \partial_x (\partial_x^2 a) - (\partial_x a)^2 =0, \quad x\in\Omega.
$$
Hence $\partial_t(\partial_x^2 a(t,x^*(t)))= (\partial_x a(t, x^*(t)))^2$ along the characteristics. Since by Lemma \ref{lem:propagation-zero-set-derivatives} the set $\{\partial_x a=0\}$ is transported along characteristics, we obtain $\partial_t(\partial_x^2 a(t,x^*(t)))=0$ on characteristics starting in $\{\partial_x a_0=0\}$. This implies the desired result.
\end{proof}

As an immediate corollary of Lemmas \ref{lem:propagation-maximum}, \ref{lem:propagation-zero-set-derivatives} and \ref{lem:conservation-second-derivative-zero-set-derivatives} we obtain

\begin{cor} \label{cor:propagation-maxima}

Assume that $a\in C^2([0,T],[-\pi,\pi])$ solves \eqref{eq-main}. Assume that $ a_0$ attains its maximum at an inner point $x^*_0$. Then for all $t\in [0,T]$, the function $x\mapsto a(t,x)$ attains its maximum at the characteristics $x^*$ starting from $x^*_0$. We have in addition $\partial_x a(t,x^*(t))=0$ and $\partial_x^2 a(t,x^*(t))=\partial_x^2 a_0(x^*_0)$ for all $t\in [0,T]$.

\end{cor}


\section{Weighted linear damping}\label{sec3}
This section is devoted to exponential decay-in-time estimates for some evolution equations related to the linearized dynamics. The linearized operator around $\cos(x)$ is
\begin{equation} \label{id:mathcalL0}
\mathcal{L}_0(\veps)= 2\cos( x) \,\veps -\sin ( x)\,\partial_x \veps +\sin( x) \int_{-\pi}^x \veps(t, \bar x) \,d\bar x - \frac{2}{\pi}  \int_{-\pi}^{\pi}\cos ( x)\, \veps(t,   x)\,dx.
\end{equation}

We will show decay-in-time estimates for a related operator that can be obtained from $\mathcal L_0$ after certain transformations,  {which are performed for the full nonlinear problem, in Subsection \ref{subs4.3}}.     
This is the key to the proof of our main theorem.

\subsection{A reduced linear problem}
As a warmup, we shall first consider the following linearized operator,
$$
L(\veps)=2 \cos (x) \,\veps-\sin (x)\, \partial_x \veps + \sin(x)\, \int_0^x \veps(t, \bar x)\,d\bar x,
$$
obtained from \eqref{id:mathcalL0} by discarding the constant term and changing the integration domain in the nonlocal term. We prove time-decay  estimates for the associated evolution equation on a moving domain:
\begin{equation}\label{eq-L}
\partial_t \veps = L(\veps)~~ {\rm on}~~\Omega^*(t),  \quad \veps|_{t=0}= \veps_0~~~~ {\rm on}~~\Omega^*_0,
\end{equation}
where $\Omega^*(t)$ is the image of a  given initial domain $\Omega^*_0=[-\pi-x^*_0, \pi-x^*_0]$  by the characteristic flow ${X(\cdot,  z)}$ that is associated to  Equation \eqref{eq-L}:
\begin{equation}\label{Flow1}
\frac{d X(t, z)}{dt}=\sin(X(t, z)), \quad X(0, z)=z\in \Omega^*_0.
\end{equation}
We consider $|x^*_0|<\delta^*$  with $\delta^*>0$ small to be specified later. The flow defined by \eqref{Flow1} is explicit, and it is a diffeomorphism from $\Omega^*_0$ to $\Omega^*(t)=[-\pi-x^*(t), \pi-x^*(t)]$ with
\begin{equation}\label{def-x*-S3}
\frac{d x^*(t)}{dt}=-\sin(x^*(t))\quad{\rm i.e.}~~x^*(t)=2\arctan\left(    e^{-t}\, \tan\left(\frac{x^*_0}{2}\right)\right).
\end{equation}
Indeed,   it is easy to check that $\pm\pi-x^*(t)$   satisfy \eqref{Flow1} and $\pm\pi-x^*(t)|_{t=0}=\pm\pi-x^*_0.$

\begin{prop}\label{Prop2.1}
 For any   constant $\theta\in (0, 1),$   there exist a  positive small constant  $\delta^*$  and   a nonnegative $2\pi$-periodic $W^{1, \infty}([-\pi, \pi])$  function  $W_\theta(x)$ satisfying $W_\theta(x)=O(|x|^3)$ as $x\to 0$ and $W_\theta(x)>0$ for $x\neq 0$, such that the following holds.  Let $\veps_0\in C^{3}(\Omega^*_0)$ such that $\veps_0(x)=O(|x|^3)$ as $x\to 0$, then the unique global solution of problem \eqref{eq-L} (noted by $e^{tL}\veps_0$)  satisfies
\begin{equation}\label{eq-2.00}
\left\|\frac{e^{tL} \veps_0}{W_\theta }\right\|_{L^\infty{(\Omega^*(t))}}\leq  e^{-(1-\theta)t}~
 \left\|\frac{\veps_0}{W_\theta }\right\|_{L^\infty{(\Omega^*_0)}}.
\end{equation}
\end{prop}

 {Let us emphasize that the two important properties of the weight $W_\theta$ constructed in the proof of Proposition \ref{Prop2.1} are the fact that $W_\theta$ only vanishes at $0$ and that it vanishes at order $O(|x|^3)$.}

To prove Proposition \ref{Prop2.1}, let us first focus on the problem without the nonlocal term $\sin x\, \int_0^x \veps(t, \bar x)\,d\bar x:$ 
\begin{equation}\label{eq-L0}
\partial_t \veps = 2 \cos x \,\veps-\sin x\, \partial_x \veps=L_0 (\veps), \quad \veps(0, x)=\veps_0,\quad {\rm for}~~x\in \Omega^*(t).
 \end{equation}
For this toy model case, we have an explicit solution by the method of characteristics, as well as a comparison principle, which shows that the long-time behavior of the solution is quantitatively determined by the behavior of the initial data near the origin. 
 
 \begin{lem}\label{Le-2.2}
 Define $W_{-1}$ as a $2\pi$-periodic function given on $[-\pi, \pi]$ by 
\begin{equation}
{W}_{-1} (x)=\left\{\begin{aligned}
 &|\sin(x) |\sin^2(x/2)\qquad |x|\leq \frac{2\pi}{3},\\
 &\frac{3\sqrt{3}}{8}\qquad\qquad\qquad \quad  x\in [-\pi, - {2\pi}/{3}]\cup [{2\pi}/{3},  \pi].
 \end{aligned}\right.
\end{equation}
Then $W_{-1}\in C^1(\mathbb R)$. For any positive  $W^{1, \infty}(\R)$  function $\widetilde W$ satisfying $\sin x ~{\widetilde W}'(x)\geq 0$ on $\R,$ the 
following holds for $\delta^*>0$ small enough. Let $\veps_0\in  C^{3}(\Omega^*_0)$ such that $\veps_0(x)=O(|x|^{3})$ as $x\to 0$, then $e^{tL0}\veps_0$ the unique global solution of problem \eqref{eq-L0}, satisfies 
\begin{equation}\label{Le2.2-2}
\left\|\frac{e^{tL_0} \veps_0}{W_{-1} \widetilde W}\right\|_{L^\infty{(\Omega^*(t))}}\leq  e^{-t}~
 \left\|\frac{\veps_0}{W_{-1}\widetilde W}\right\|_{L^\infty{(\Omega^*_0)}},\end{equation}
  \end{lem}
  
 {Notice that one can of course take $\widetilde W=1$ in Lemma \ref{Le-2.2}. The factor $\widetilde W$ enables us to gain flexibility, which will be useful to handle the nonlocal terms appearing in the operator $\mathcal L_0$.}

  \begin{proof}
  The proof is based on the  comparison principle for the linear transport equation \eqref{eq-L0}.   In fact,   following the characteristics  it is not difficult to observe that Equation  \eqref{eq-L0} enjoys a comparison principle:  any $W^{1, \infty}$ supersolution  is nonnegative on $\Omega^*(t)$  if it is  nonnegative initially on $\Omega^*_0.$ 
  
   Given  $W_{-1}$ and  $\widetilde W$ as in above statement,   we claim that  $e^{-t}W_{-1}\widetilde W$ is a supersolution of \eqref{eq-L0},  i.e.
  \begin{equation}\label{eq-2.2} 
(\partial_t -L_0)(e^{-t}W_{-1}\widetilde W)\geq 0,\quad {\rm for} ~~ x\in  \Omega^*(t).
  \end{equation}
   In fact,  notice  that $W_{-1}'(\frac{2\pi}{3})=0$ ensures $W_{-1}$ is $C^1(\Omega^*(t))$. Moreover, on $[-\frac{2\pi}{3},\frac{2\pi}{3}]$ we have $L_0(W_{-1})=-W_{-1}$ by a direct computation. Hence $ (\partial_t -L_0)(e^{-t}W_{-1})=0$ for $|x|\leq \frac{2\pi}{3}$,  so we  do have
       \begin{equation*} (\partial_t -L_0)(e^{-t}W_{-1}\widetilde W)=    e^{-t}W_{-1}\sin (x)\widetilde W'\geq 0,\quad {\rm for} ~~|x|\leq \frac{2\pi}{3}.
 \end{equation*}
 
Now,    for fixed time $t,$  we are going to  show \eqref{eq-2.2} is valid  on the remaining domain  $x\in [-\pi-x^*(t), -\frac{2\pi}{3}]\cup [\frac{2\pi}{3}, \pi-x^*(t)].$   {Notice that one of the two intervals may be empty.}   
To do that,  we find that the origin is a fixed point of  \eqref{def-x*-S3} and it is stable:    $|x^*(t)|$ is small enough  under the assumption that  $|x^*_0|<\delta^*$ is small enough.
Our definiton of $W_{-1}$ then yields 
 \begin{equation*}
{W}_{-1} (x)=
\frac{3\sqrt{3}}{8}\quad {\rm for} ~~ x\in [-\pi-x^*(t), - {2\pi}/{3}]\cup [{2\pi}/{3}, \pi-x^*(t)].
\end{equation*}
This together with the inequality  $\cos(x)\leq -\frac{1}{2}$ on $[-\pi-x^*, -\frac{2\pi}{3}]\cup [\frac{2\pi}{3}, \pi-x^*]$ implies that 
 \begin{equation*} (\partial_t -L_0)(e^{-t}W_{-1}\widetilde W)=   \frac{3\sqrt{3}}{8}\widetilde We^{-t} \left(-1-2\cos(x)\right)   +\frac{3\sqrt{3}}{8}e^{-t} \sin(x)    \widetilde W'\geq 0,
 \end{equation*}
and there holds inequality \eqref{eq-2.2}.
\medskip

Defining  $\left\|\frac{\veps_0}{W_{-1}\widetilde W}\right\|_{L^\infty{(\Omega^*_0)}}=:C,$ we have by definition $-C W_{-1}\widetilde W\leq \veps_0\leq C W_{-1}\widetilde W$ on $\Omega^*_0.$  Thanks to \eqref{eq-2.2}, we see that $Ce^{-t}W_{-1}\widetilde W\pm e^{tL_0}\veps_0$  are  supersolutions  of Equation \eqref{eq-L0}  supplemented with nonnegative initial data on $\Omega^*_0.$  Therefore,  \eqref{Le2.2-2} is a consequence of the comparison principle.
 \end{proof}
 
 Next, we have the following lemma, which indicates that the nonlocal term in \eqref{eq-L} can be treated perturbatively.
 
 \begin{lem}\label{Le-2.3}
 Let $W_{-1}$ be defined as in Lemma \ref{Le-2.2}.     For any   constant $\theta\in(0, 1)$ and  any  $x^*  \in (-\delta^*, \delta^*),$    if   $\delta^*$ is small enough  then  there exists a positive $W^{1, \infty}(\mathbb R)$ function  $\widetilde{W}_\theta$ with $\sin x \widetilde{W}_\theta'\geq 0$ such that   the following estimate holds:    
\begin{equation}\label{lemeq-3.3}
\left\| \frac{\sin x \int_0^x f(\bar x)\,d\bar x}{W_{-1}\widetilde W_\theta}\right\|_{L^\infty( [-\pi-x^*, \pi-x^*])}\leq ~\theta ~\left\|\frac{f}{W_{-1}\widetilde W_\theta}\right\|_{L^\infty([-\pi-x^*, \pi-x^*])}.
 \end{equation}
 \end{lem}
 \begin{proof}
For any $\theta \in(0, 1)$ and  $C>0$ (to be chosen later),   we define  $\widetilde{W}_\theta$ a 2$\pi-$periodic function whose definition on  $[-\pi, \pi]$ is given  by
 \begin{equation}
 \widetilde{W}_\theta (x) =
 \exp(C|x|/{\theta}),\quad |x|\leq \pi. 
 \end{equation}

 Define $v(x)=\frac{f(x)}{W_{-1}\widetilde{W}_{\theta}},$ since $W_{-1} \widetilde{W}_\theta \geq 0$  and both $W_{-1},~\widetilde{W}_\theta$ are even, we only need to discuss about the case when $x\in[0, \pi-x^*]$  and find that 
   \begin{align*}
\left|\frac{\sin x \int_0^x f(\bar x)\,d\bar x}{W_{-1}\widetilde W_\theta}\right|  \leq \|v\|_{L^\infty(\Omega^*)} \frac{ |\sin (x)| \int_0^x W_{-1} \widetilde{W}_\theta \,d\bar x}{W_{-1} \widetilde{W}_\theta}, \quad x\in [0, \pi-x^*].
\end{align*}

  At first, for $x\in[0,  \frac{2\pi}{3}]$ one has 
 \begin{align*}
\frac{ |\sin (x)| \int_0^x W_{-1} \widetilde{W}_\theta \,d\bar x}{W_{-1} \widetilde{W}_\theta}&\leq  \frac{ \int_0^x \sin^2(\bar x/2)\,\sin \bar x \exp(C\bar x/\theta)  \,d\bar x}{ \sin^2({x/2})\exp(C\theta^{-1} \sin^2(x/2))}  \\
  &\leq  \frac{  \int_0^{x} \exp(C \bar x/\theta)  \,d \,\bar x}{ \exp(Cx\theta)}  \\
&\leq \frac{\theta}{C}.
 \end{align*}
   
Now, we discuss the case  when  $0\leq x^*<\delta^*.$ Thanks to  the fact that $0\leq W_{-1}(x)\leq \frac{3\sqrt{3}}{8}$ for all $x\in\Omega^*$ and $W_{-1}\equiv  \frac{3\sqrt{3}}{8}$ for $x\geq \frac{2\pi}{3},$     we see that  when $x\in[\frac{2\pi}{3}, \pi-x^*],$ 
     \begin{align*}
    \frac{ |\sin (x)| \int_0^x W_{-1} \widetilde{W}_\theta \,d\bar x}{W_{-1} \widetilde{W}_\theta}\leq
\frac{\int_0^x  \exp(C\bar x/\theta)\,d\bar x}{\exp(C x/\theta)}\leq \frac{\theta}{C}.
 \end{align*}
As to the case $-\delta^*<x^*\leq0,$  similarly,  for $x\in[\frac{2\pi}{3}, \pi]$ we can write that
 \begin{align*}
    \frac{ |\sin (x)| \int_0^x W_{-1} \widetilde{W}_\theta \,d\bar x}{W_{-1} \widetilde{W}_\theta}\leq
\frac{\int_0^x  \exp(C\bar x/\theta)\,d\bar x}{\exp(C x/\theta)}\leq \frac{\theta}{C},
 \end{align*}
while for $x\in [\pi, \pi-x^*],$ one has 
  \begin{align*}
      \int_0^x W_{-1} \widetilde{W}_\theta \,d\bar x\leq  \frac{3\sqrt{3}}{8}\left(  \int_0^\pi\exp(C\bar x/\theta)\,d\bar x +  \int_\pi^x \exp(C(2\pi-\bar x)/\theta)\,d\bar x\right),
       \end{align*}
and thus
\begin{align*}
    \frac{ |\sin (x)| \int_0^x W_{-1} \widetilde{W}_\theta \,d\bar x}{W_{-1} \widetilde{W}_\theta}\leq&
 \frac{2\theta \exp(C\pi/\theta)}{C\exp(C(2\pi-x)/\theta)}\\
\leq& \frac{2\theta}{C}\exp(Cx^*/\theta)\leq \frac{6\theta}{C},
 \end{align*}
 provided that $\delta^*\leq \frac{\theta}{C}.$
 Therefore, for all $x\in[0, \pi-x^*]$ and $C$ large enough we have 
 \begin{align*}
\left|\frac{\sin ( x) \int_0^{x}f(\bar x)\,d\bar x}{W_{-1}\widetilde W_\theta}\right|\leq \theta \|v\|_{L^\infty([-\pi-x^*, \pi-x^*])},
 \end{align*}
and there holds \eqref{lemeq-3.3}.
 \end{proof}
 
\bigskip

We are ready to prove Proposition \ref{Prop2.1}. We choose the following weight
 \begin{equation} \label{id:def-Wtheta}
 W_\theta=W_{-1}\widetilde W_\theta.
 \end{equation}
where $W_{-1}$ and $\widetilde{W}_{\theta}$ are as in Lemma \ref{Le-2.2} and Lemma \ref{Le-2.3}.

\begin{proof}[Proof of Proposition \ref{Prop2.1}]

Applying Duhamel's formula to   problem  \eqref{eq-L} yields that
 \begin{equation}\label{ }
\veps = e^{tL} \veps_0 =  e^{tL_0} \veps_0 +\int_0^t e^{(t-\tau)L_0}\left(\sin(x)\int_0^x\veps(\tau, \bar x)\,d\bar x\right) \,d\tau.
\end{equation}
Since $\sin x \,\widetilde{W}_{\theta}'(x)\geq0$, one can apply   Lemma \ref{Le-2.2}.  This together with Lemma \ref{Le-2.3}  yields  that
  \begin{align*}
 \left \|\frac{ \veps(t, \cdot)}{W_\theta}\right\|_{L^\infty(\Omega^*(t))}  \leq& e^{-t} \left\|\frac{\veps_0}{W_\theta} \right\|_{L^\infty(\Omega^*(0))}+\int_0^t e^{(\tau-t)} \left \|\frac{\sin(x)\int_0^x \veps(\tau, \bar x)\,d\bar x}{W_\theta}\right\|_{L_x^\infty(\Omega^*(\tau))}\,d\tau\\
\leq&  e^{-t} \left \|  \frac{\veps_0}{W_\theta}\right\|_{L^\infty(\Omega^*(0))}+\theta \int_0^t e^{(\tau-t)} \left\|\frac{\veps (\tau)}{W_\theta}\right\|_{L_x^\infty(\Omega^*(\tau))}\,d\tau,
\end{align*}
which finally  implies \eqref{eq-2.00} thanks to Gronwall's lemma.   
 \end{proof}
 
 \bigbreak
 
 \subsection{The quasilinearized problem}\label{subsec.quasilin}
 At this stage,  we are ready to establish  a key  exponential time-decay estimate for  the following quasilinearized transport equation with nonlocal term
\begin{equation}\label{eq:quasi:L}
\left\{\begin{aligned} 
&\partial_s u - L (u) + \left(\int_{0}^y \eta(s,\bar y) \,d\bar y\right) \partial_y u =0\quad{\rm on}~~ [s_0,s_1]\times \Omega^*(s)\\
&u(s_0, y)=u_0(y)\quad{\rm on}~~   \Omega^*_{s_0},
\end{aligned}\right.
\end{equation}
where  the moving domain $\Omega^*(s)$ is the image of $\Omega^*_{s_0}:=[-\pi -x^*_{s_0}, \pi-x^*_{s_0}]$ by the characteristic flow associated to \eqref{eq:quasi:L}:
 \begin{equation}\label{Flow2}
\frac{d Y(s, z)}{ds}=\sin(Y(s,  z)) +\int_0^{Y(s, z)} \eta(s, \bar y)\, d\bar y,\quad Y(s_0, z)=z\in \Omega^*_{s_0}.
\end{equation}

  Note that   $x^*(s)$ given by
\begin{equation}\label{def-x*-Q-S3}
\frac{d x^*(s)}{ds}=-\sin(x^*(s))  + \int_{-\pi-x^*(s)}^0\eta(s, \bar y)\, d\bar y,  \quad x^*(s)|_{s=s_0}=x^*_{s_0}
  \end{equation}
  ensures  $\pm\pi-x^*(s)$   satisfy \eqref{Flow2}.    Therefore $\pm\pi-x^*(s)$ is the image of $\pm\pi-x^*_{s_0}$ by the characteristic flow \eqref{Flow2},  and there holds $\Omega^*(s)=[-\pi-x^*(s),  \pi-x^*(s)]$.

We term the problem \eqref{eq:quasi:L} as ``quasilinearized" since later on the function $\eta$ will actually depend on the solution. The equation  considered here is the full linearized equation of  the nonlinear problem \eqref{eq-vareps} to which one adds the nonlinear transport term,  after performing a suitable change of variables and modulation analysis, leading to Equation \eqref{eq-xi}.


    We introduce the solution map $S(s,s_0)$ associated to  \eqref{eq:quasi:L},
that is to say  for given $0\leq s_0 \leq s_1$ and $u_0$,  then $s\mapsto S(s,s_0)u_0$ is the solution to \eqref{eq:quasi:L}.

\begin{prop}\label{Prop-quasi}
Let $W_\theta$ be given by \eqref{id:def-Wtheta}. We keep the notations of Lemmas   \ref{Le-2.2},  \ref{Le-2.3} and assume the following assumptions  on a time interval $[0, s_1]:$
 \begin{align}
|x^*(s)|\leq& \delta^*, \label{bd:Leq2}\\
  \|\eta(s, \cdot)\|_{L^{\infty}(\Omega^*)}\leq& \delta^*.\label{bd:Leq3}
 \end{align} 

 Then, for any $0<\theta<\theta'<1,$ if $\delta^*>0$ is small enough, we have for all $s\in [s_0,s_1]:$
\begin{equation}\label{eq-Prop3.4}
\left\|\frac{S(s,s_0)u_0}{W_\theta}\right\|_{L^\infty{(\Omega^*(s))}}\leq  e^{-(1-\theta')(s-s_0)}~
 \left\|\frac{u_0}{W_\theta}\right\|_{L^\infty{(\Omega^*_{s_0})}}.
\end{equation}

\end{prop}
   
 Keeping  Lemma \ref{Le-2.3} in mind,   similarly to the proof of  Proposition \ref{Prop2.1}, we start by showing the above statement by replacing the nonlocal operator $L$ by the associated  local operator $L_0.$

   \begin{lem} \label{Le-quasi}
Under the assumptions of Proposition \ref{Prop-quasi} we have
$$
\left\|\frac{\bar{S}(s,s_0)u_0}{W_\theta}\right\|_{L^\infty{(\Omega^*(s))}}\leq  e^{-(1-\theta')(s-s_0)}~
 \left\|\frac{u_0}{W_\theta}\right\|_{L^\infty{(\Omega^*_{s_0})}}.
$$
where $\bar S(s,s_0)u_0$ denotes the solution to
$$
\left\{\begin{aligned} 
&\partial_s u - L_0 (u) + \left(\int_{0}^y \eta(s,\bar y) \,d\bar y\right) \partial_y u =0\quad{\rm on}~~ [s_0,s_1]\times \Omega^*(s)\\
&u(s_0, y)=u_0(y)\quad{\rm on}~~   \Omega^*_{s_0},
\end{aligned}\right.
$$

\end{lem}

   \begin{proof}
   The proof is based on the comparison principle along characteristics.  Define $ \left\|\frac{u_0}{W_\theta}\right\|_{L^\infty{(\Omega^*_{s_0})}} =:C,$ which yields $-C \,W_\theta\leq u_0\leq C\, W_\theta$ on $\Omega^*_{s_0}.$  In what follows,   for any $0<\theta<\theta'<1,$ we  want to show  that $e^{-(1-\theta')s}W_\theta$  is a supersolution for suitable  small  constant $\delta^*$  {(see \eqref{bd:Leq2} and \eqref{bd:Leq3})},
where  $W_\theta=W_{-1}\, \widetilde{W}_{\theta}$  is taken  exactly the same as in  the proof of Proposition \ref{Prop2.1}.  We write that 
    \begin{multline}\label{Le5-000}
  \left(\partial_s  - L_0 + \left(\int_{0}^y \eta(s,\bar y) \,d\bar y\right) \partial_y \right) \left(e^{-(1-\theta')s}W_\theta \right)  \\
   \qquad=e^{-(1-\theta')s} \left( \frac{\theta'}{2} W_{-1} + \left(\int_{0}^y \eta(s,\bar y) \,d\bar y\right){W}_{-1}' \right) \widetilde{W}_\theta \\
   \qquad  \qquad+e^{-(1-\theta')s} \left( \frac{\theta'}{2} \widetilde{W}_\theta  + \left(\int_{0}^y \eta(s,\bar y) \,d\bar y\right)  \widetilde{W}_{\theta}' \right)W_{-1}  \\
  \qquad  + e^{-(1-\theta')s}  \sin(y)W_{-1} \widetilde{W}_{\theta}'.
  \end{multline}
  
Thanks to the symmetry properties of $W_{-1}$ and $\widetilde{W}_{\theta},$ we will only discuss about the case $y\in[0, \pi-x^*]$ and $-\delta^* <x^*(s)\leq0,$    for any fixed $s\in[s_0,  s_1].$ 
 Noticing  that
 \begin{equation*}
| W_{-1}'|=|\cos(y)\sin^2(y/2)+\sin(y)\sin(y/2)\cos(y/2)|\leq 3\sin^2(y/2).
 \end{equation*}
By virtue of  \eqref{bd:Leq3} and by the fact that $\sin (y)\geq \frac{y}{\pi}$ for  all  $y\in[0, \frac{2\pi}{3}], $    we have
\begin{align}
\frac{\theta'}{2} W_{-1} + \left(\int_{0}^y \eta(s,\bar y) \,d\bar y\right){W}_{-1}' \geq& \left(\frac{\theta'}{2}\sin (y)-6cy\right)\sin^2(y/2)\notag\\
\geq& \left(\frac{\theta'}{2\pi} -6c\right)y\sin^2(y/2)\geq 0,
  \end{align}
 whenever $\delta^*\leq \frac{\theta'}{12\pi}.$ While for $y\in(\frac{2\pi}{3}, \pi-x^*],$ as $W_{-1}'\equiv0,$ we still have 
 \begin{align*}
 \frac{\theta'}{2} W_{-1} + \left(\int_{0}^y \eta(s,\bar y) \,d\bar y\right){W}_{-1}' \geq0.
\end{align*}

In order to consider the second part of the right-hand-side of \eqref{Le5-000}, we write that for $y\in[0, \pi]$ that
\begin{align}
 \frac{\theta'}{2} \widetilde{W}_\theta  + \left(\int_{0}^y \eta(s,\bar y) \,d\bar y\right)  \widetilde{W}_{\theta}'&= \frac{\theta'}{2} e^{\frac{C}{\theta} y}+ \frac{C}{\theta}  \left(\int_{0}^y \eta(s,\bar y) \,d\bar y\right) e^{\frac{C}{\theta} y}\notag\\
 &\geq \left( \frac{\theta'}{2}-\frac{2C\delta^*\pi}{\theta}\right)e^{\frac{C}{\theta} y}\geq 0,
\end{align}
provided that $\delta^* \leq \frac{\theta \theta'}{4C\pi}.$
   Similarly,  for $y\in[\pi, \pi-x^*]$,
 \begin{align}
 \frac{\theta'}{2} \widetilde{W}_\theta  + \left(\int_{0}^y \eta(s,\bar y) \,d\bar y\right)  \widetilde{W}_{\theta}'&= \frac{\theta'}{2} e^{\frac{C}{\theta} (2\pi-y)}+ \frac{C}{\theta}  \left(\int_{0}^y \eta(s,\bar y) \,d\bar y\right) e^{\frac{C}{\theta} (2\pi-y)}\notag\\
 &\geq \left( \frac{\theta'}{2}-\frac{4C\delta^*\pi}{\theta}\right)e^{\frac{C}{\theta} (2\pi-y)}\geq 0,
\end{align}
 provided that $\delta^* \leq \frac{\theta \theta'}{8C\pi}.$ 
 Finally,  due to the fact that $\sin(y)\widetilde{W}_{\theta}'\geq0,$ we conclude from above estimates that
 \begin{equation}
  \left(\partial_s u - L_0 + \left(\int_{0}^y \eta(s,\bar y) \,d\bar y\right) \partial_y \right) \left(e^{-(1-\theta')s}W_\theta \right) \geq 0, \quad {\rm for~ all}~~y\in[0, \pi-x^*].
 \end{equation}

This shows that  $e^{-(1-\theta')(s-s_0)}CW_\theta\pm \bar{S}(s, s_0)u_0$ are   supersolutions supplemented with nonegative initial data on $\Omega^*_{s_0}.$  Consequently,  we infer from  the comparison principle that these supersolutions are nonnegative on $\Omega^*(s)$ after time evolution,  which completes the proof.
  \end{proof}

 \begin{proof}[Proof of Proposition \ref{Prop-quasi}]
 
 It can be proved as a consequence of Lemma \ref{Le-quasi} and of the Duhamel formula, exactly as Proposition \ref{Prop2.1} was proved as a consequence of Lemma \ref{Le-2.3}.
 \end{proof}

\section{Reformulation of the nonlinear problem}\label{S4}
 
 To prove the main theorem, we need to do some preparatory work to convert the original problem in a formulation that is covered by the reduced problem discussed in Section \ref{sec3}. We recall that Equation \eqref{eq-main} has the scaling symmetry \eqref{id:scaling-symmetry}. Therefore, we will only prove the stability of the stationary solution $\cos(\cdot)$, which will imply the stability of the other stationary solutions $\mu \cos(\cdot)$ for $\mu>0$.
 


\subsection{Renormalization and equation in self-similar variables}

We will perform a change of variables consisting in a rescaling and a translation in space. 

Firstly, due to the scaling invariance \eqref{id:scaling-symmetry}, we will perform a rescaling to convert the stability problem for the full family of stationary states $\{\mu \cos(x)\}_{\mu>0}$ to the stability problem for the fixed function $\cos(x)$.

Secondly, the transport field $\int_{-\pi}^x a(t, \bar x) \,d\bar x$ in \eqref{eq-main} vanishes at the boundary $x=\pm\pi$, so that $x=\pm\pi$ is a fixed point of the characteristics. We will perform a translation in space so that in the new variables, it is the point at which the solution attains its maximum that is a fixed point of the new characteristics. We recall that by Lemma \ref{lem:propagation-maximum} the maximum of solutions to \eqref{eq-main} is transported along the characteristics.

To fix the point at which the solution $a$ attains its maximum, we shall use that Equation \eqref{eq-main} is almost invariant by translation in space. That is to say, all terms except $\int_{-\pi}^x  a(t, \bar x) \,d\bar x$ in \eqref{eq-main} are invariant by spatial translations.

Let $a$ solve \eqref{eq-main}. For two $C^1$ functions of time  $\mu=\mu(s)>0$ and $x^*=x^*(s)\in \mathbb R$ we introduce the change of variables
\begin{equation} \label{self-similar-variables}
a(t,x)=\mu(s)A(s,y), \qquad s=\int_0^t \mu(t')dt', \qquad y=x-x^*(s).
\end{equation}
Then Equation \eqref{eq-main} becomes
$$
\partial_s A+\frac{\dot{\mu}}{\mu} A +\left( \int_{-\pi-x^*}^y A(s,\bar y)d\bar y-\dot{x}^* \right)\partial_y A-A^2+\frac{1}{\pi}\int_{-\pi-x^*}^{\pi+x^*} A^2(s,\bar y)d\bar y=0
$$
where we use a dot to denote differentiation in time $s$, i.e. $\dot{\mu} =\frac{d}{ds}\mu$ and $\dot{x}^* =\frac{d}{ds}x^*$. 
We define $x^*$ to solve
\begin{equation} \label{modulation-equation-x*}
\dot{x}^* =\int_{-\pi-x^*}^0 A(s,\bar y)d\bar y.
\end{equation}
We still denote by $A$ the $2\pi$-periodic extension of $A$ to the whole line $\mathbb R$. Note that $A$ might have discontinuities at $\pi-x^*$ and $-\pi-x^*$.

Using \eqref{modulation-equation-x*} and that $A$ is a 2$\pi$-periodic function, we get that the new renormalized unknown $A$ solves 
\begin{equation}\label{equation-after-renormalization}
\partial_s A+\frac{\dot{\mu}}{\mu} A +\left(\int_{0}^y A(s,\bar y)d\bar y\right)\partial_y A-A^2+\frac{1}{\pi}\int_{-\pi}^\pi A^2(s,\bar y)d\bar y=0, \quad {\rm for}~~ y\in\Omega^*,
\end{equation}
where $\Omega^*(s)=[-\pi-x^*, \pi-x^*]$ is the moving interval. The function $A$ satisfies the mean-free condition in renormalized variables
\begin{equation} \label{bc-renormalized-variables}
\int_{-\pi}^\pi A(s,y)dy=0.
\end{equation}
We remark that the transformation \eqref{self-similar-variables} actually does not necessitate $x^*$ to be the maximum of $a$. But this is what the reader should have in mind for the sequel.

\subsection{Linearization and spectral analysis}

We note that for each $\mu_0>0$,  the profile $(\mu(s), x^*(s), A(s,y))=(\mu_0, 0, \cos(y))$ produces a stationary solution to the renormalized equations \eqref{modulation-equation-x*}-\eqref{equation-after-renormalization}, which corresponds to the family of stationary states $\{{\mu_0} \cos(x)\}_{{\mu_0}>0}$ in original variables. We now decompose
$$
A(s,y)=\cos(y)+\eta(s,y)
$$
and the equations \eqref{modulation-equation-x*}-\eqref{equation-after-renormalization}-\eqref{bc-renormalized-variables} become
\begin{equation}\label{eq-x*}
 \dot{x}^*  = -\sin x^* +\int_{-\pi-x^*}^0 \eta(\bar y)\,d \bar y,
 \end{equation}
and
\begin{align}\label{eq-eta}
& \partial_s \eta - \mathcal{L}(\eta)+\frac{\dot{\mu}}{\mu}\left(\cos(y)+\eta \right) +  \left(\int_{0}^y \eta(t, \bar y) \,d\bar y\right) \partial_y \eta  -\eta^2 +\frac{1}{\pi}  \int_{-\pi}^{\pi}\eta^2(\bar y)\,d\bar y=0, \\
& \label{bc-renormalized-variables-2} \int_{-\pi}^\pi \eta(s,y)dy=0, 
\end{align}
where the associated linear operator $\mathcal{L}$ takes the form
\begin{equation*}
\mathcal{L}(\eta)= 2\cos y \,\eta -\sin y\,\partial_y \eta +\sin y \int_{0}^y \eta( \bar y) \,d\bar y - \frac{2}{\pi}  \int_{-\pi}^{\pi}\cos \bar y\, \eta(  \bar y)\,d\bar y, \quad y\in\Omega^*,
\end{equation*}
and the initial conditions are
\begin{align}
&\mu(s)|_{s=0}:=\mu_0,\\
&x^*(s)|_{s=0}:=x^*_0,\\
&\eta(s, y)|_{s=0}:=\eta_0(y)=  \frac{1}{\mu_0}a_0(x^*_0+y)-\cos (y). \label{definition-eta0}
\end{align}
As $A$ and $\cos$ are both $2\pi$ periodic, notice that $\eta$ is also $2\pi$ periodic, with possible discontinuities at $\pi-x^*$ and $-\pi-x^*$.

We now give some explicit eigenmodes of the linear operator $\mathcal{L},$ that will play an important role for our analysis later on. They are obtained from the following algebraic identities.

 \begin{lem}\label{Le-3.1}
The following identities hold
 \begin{align*}
&  \mathcal L(\cos(y))=0,\qquad  \mathcal L(1)=2\cos (y)+y \sin (y),\qquad \mbox{and} \qquad \mathcal L(y\sin(y))=1.
 \end{align*}
In particular,
\begin{align*}
& \varphi_{-1}(y)=-2\cos(y)+1-y\sin (y) ,\\
&  \varphi_{0}(y)=\cos(y) ,\\
&  \varphi_{1}(y)=2\cos(y)+1+y\sin(y),
\end{align*}
are eigenfunctions of $\mathcal L$ in the sense that
$$
\mathcal L \varphi_l= l\varphi_l, \qquad l=-1,0,1.
$$


\end{lem}

\begin{proof}
 
An explicit computation shows $\mathcal L (1)=2\cos y +0+y\sin y +0$. It is then instructive to consider how the two other functions $\cos(y)$ and $y\sin(y)$ appear from a symmetry and an almost symmetry of the stationary equation. Namely, for all $\mu \in \mathbb R$ and $\lambda>0$, the function $a_{\mu,\lambda}(x)= \mu \cos(\lambda x){\mathbf{1}_{(-\frac{\pi}{\lambda}\leq x \leq \frac{\pi}{\lambda})}}$ solves the modified stationary equation
$$
\left(\int_0^y a_{\mu,\lambda}(\bar y)d\bar y\right)\partial_y a_{\mu,\lambda}- a_{\mu,\lambda}^2+\frac{\lambda}{\pi} \int_{\mathbb R} a_{\mu,\lambda}^2(\bar y)d\bar y =0
$$
on its support. Differentiating the above equation with respect to $\mu$ at $(\mu,\lambda)=(1,1)$ gives $\mathcal L(\cos y)=0$. Differentiating with respect to $\lambda$ gives $\mathcal L(y\sin y)=1$.
\end{proof}

We remark the following Taylor expansions at the origin,
\begin{align} \label{taylor-origin-eigenfunctions}
 \varphi_{-1}(y)= -1+O(y^4), \qquad  \varphi_{0}(y)=1-\frac{y^2}{2}+O(y^4), \qquad \varphi_{1}(y)=3+O(y^4)
\end{align}
as $y\to 0$. We also compute the mean of the eigenfunctions
\begin{align} \label{mean-eigenfunctions}
\int_{-\pi}^\pi \varphi_{-1}(y)dy=0, \qquad \int_{-\pi}^\pi \varphi_{0}(y)dy=0, \qquad \frac{1}{2\pi}\int_{-\pi}^\pi \varphi_{1}(y)dy=2.
\end{align}

\begin{rmk}
Although Lemma \ref{Le-3.1} states that there exists a linear instability, we notice that by \eqref{mean-eigenfunctions} the associated eigenfunction $\varphi_{1}$ does not satisfy the zero mean condition as $\frac{1}{2\pi}\int_{-\pi}^\pi \varphi_{1}(y)dy=2$. Therefore, this is not an instability in the set of mean-free functions.
\end{rmk}

\subsection{Decomposition of the perturbation and modulation equations}\label{subs4.3}

We introduce the projection operator
$$
\mathcal{P}(\eta)=   \frac{1}{\pi}  \int_{-\pi}^{\pi}\cos( y) \ \eta(y) \,dy .
$$
This term is not easy to handle in $\mathcal L$, as it is very nonlocal. We will be able to remove it thanks to modulating the perturbation along the explicit eigenfunctions in Lemma \ref{Le-3.1}. Namely, we decompose the constant function $1$ in the $\{\varphi_l\}_{l=-1,0,1}$ basis
$$
1 = \frac{1}{2} (\varphi_{-1}+\varphi_1)
$$
so that the above projection operator is a projection onto a one-dimensional subset of $\textup{Span}(\{\varphi_l\}_{l=-1,0,1})$,
$$
\mathcal{P}(\eta)=  \frac 12 \mathcal{P}(\eta) \varphi_{-1}+\frac 12 \mathcal{P}(\eta) \varphi_1.
$$
We now look for a decomposition
$$
\eta=\xi+ \alpha_{-1}\varphi_{-1}+\alpha_1 \varphi_1
$$
for some time-dependent coefficients $\alpha_{-1}=\alpha_{-1}(s)$ and $\alpha_{1}=\alpha_1(s)$. The linearized operator now takes the form
$$
\mathcal L(\eta) = \left( -\alpha_{-1}-\mathcal{P}(\xi) \right)\varphi_{-1}+\left( \alpha_1 - \mathcal{P}(\xi)\right) \varphi_1 +L( \xi) .
$$
We introduce
\begin{equation}\label{e.defbeta}
\beta=\alpha_{-1}\varphi_{-1}+\alpha_1 \varphi_1,\quad \boldsymbol{\alpha}=(\alpha_{-1}, \alpha_1).
\end{equation}
Notice that $\xi$ is not a $2\pi$-periodic function anymore, but is the sum of a $2\pi$ periodic function with possible discontinuities at $\pi-x^*$ and $-\pi-x^*$, and of a smooth function on $\mathbb R$ that is not $2\pi$-periodic. More precisely, we introduce $\mathcal E^s$ as the operator that extends functions on $\Omega^*(s)$ to $2\pi$-periodic functions on the whole line $\mathbb R$. Then the function $\xi$ on the whole line $\mathbb R$ can be retrieved from its restriction $\xi_{|\Omega^*(s)}$ to $\Omega^*(s)$ via the formula
$$
\xi =\mathcal E^s (\xi_{|\Omega^*(s)})+ \mathcal E^s (\beta_{|\Omega^*(s)})-\beta.
$$
Noticing that for $l=-1,1$, the function $\mathcal E^s (\varphi_{l|\Omega^*(s)})-\varphi_{l}$ restricted to the set $[-\pi,\pi]$ is only non-zero on the set $[-\pi,\pi]\backslash \Omega^*(s)$ that has measure $\lesssim |x^*|$, we have the estimate
\begin{equation} \label{bd:xi-L1--pi-pi}
\| \xi \|_{L^1([-\pi,\pi])} \lesssim \| \xi\|_{L^\infty(\Omega^*(s))}+|x^*||\boldsymbol{\alpha}|.
\end{equation}
Similarly, in equation  \eqref{eq-eta} the quadratic term
$$
\mathcal Q (\eta):=\frac{1}{2\pi} \int_{-\pi}^{\pi}\eta^2(\bar y)\,d\bar y
$$
is not easy to handle, and we will remove it by projecting $2\mathcal Q (\eta)= \mathcal{Q}(\eta) \varphi_{-1}+ \mathcal{Q}(\eta) \varphi_1$. The evolution equation \eqref{eq-eta} becomes
\begin{multline}\label{eq-eta-inter}
  \partial_s \xi - L( \xi) +\frac{\dot\mu }{\mu} \xi + \left(\int_{0}^y \eta(s,  \bar y) \,d\bar y\right) \partial_y \xi -\left(2\beta+\xi \right)\xi\\
 +\left(\dot\alpha_{-1}+\alpha_{-1}+\frac{\dot\mu}{\mu}\alpha_{-1}+\mathcal P(\xi)+\mathcal Q(\eta)\right)\varphi_{-1}+\frac{\dot\mu}{\mu} \varphi_0\\
 +\left(\dot\alpha_{1}-\alpha_{1}+\frac{\dot\mu}{\mu}\alpha_{1}+\mathcal P (\xi)+\mathcal Q (\eta)\right)\varphi_{1}\\
 +  \left(\int_0^y \eta(s,  \bar y)\,d\bar y\right)\partial_y \beta-\beta^2=0
\end{multline}
for $y\in \Omega^*(s)$.

Thanks to Proposition \ref{Prop2.1}, we know that the semigroup $e^{tL}$ decays exponentially fast for functions that vanish up to second order at the origin. We will therefore ensure that $y\mapsto \xi(t,y)$ vanishes up to second order near the origin by choosing suitably the evolution equations for $\mu$, $\alpha_{-1}$ and $\alpha_1$. For this purpose, we investigate the behavior near the origin of the nonlinear terms in the fourth line of \eqref{eq-eta-inter}. We have by \eqref{taylor-origin-eigenfunctions}
$$
\left(\int_{0}^y \eta(s,  \bar y) \,d\bar y  \right) \partial_y  \beta =O(y^4)
$$
and according to the definition of $\beta$ in \eqref{e.defbeta}
\begin{equation}\label{e.betaest}
\beta^2=\alpha_{-1}^2-6\alpha_{-1}\alpha_1+9\alpha_1^2+O({|\boldsymbol{\alpha}|^2}y^4)
\end{equation}
as $y\to 0$. We remark that the behavior above ressembles the Taylor expansion of $\varphi_{-1}$ near the origin. Hence we decompose
$$
\beta^2=  \mathcal{Q}_1(\boldsymbol{\alpha})\varphi_{-1}+\mathcal N_1(s, y),
$$
 where the quadratic form and nonlinear remainder term are
\begin{align*}
& \mathcal{Q}_1(\boldsymbol{\alpha}):=  \alpha_{-1}^2-6\alpha_{-1}\alpha_1+9\alpha_1^2,\\
& \mathcal{N}_1(s, y):=\beta^2-  \mathcal{Q}_1(\boldsymbol{\alpha})\varphi_{-1}.
\end{align*}
According to \eqref{e.betaest} the nonlinear term satisfies
\begin{equation}\label{taylorexpansionN}
|\mathcal{N}_1(s, y)|=O(|\boldsymbol{\alpha}|^2 y^4) \quad {\rm as}~~y\to 0.
\end{equation}
We remark that if $\xi(y)=O(|y|^3)$ near the origin, then all terms in the first line of \eqref{eq-eta-inter} are $O(|y|^3)$ as well. We thus rewrite \eqref{eq-eta-inter} as
  \begin{multline}\label{eq-eta-inter-2}
 \partial_s \xi - L (\xi) +\frac{\dot\mu }{\mu} \xi + \left(\int_{0}^y \eta(s,  \bar y) \,d\bar y\right) \partial_y \xi -\left(2\beta+\xi \right)\xi +  \left(\int_0^y \eta(s,  \bar y)\,d\bar y\right)\partial_y \beta- \mathcal{N}_1\\
 +\left(\dot\alpha_{-1}+\alpha_{-1}+\frac{\dot\mu}{\mu}\alpha_{-1}+\mathcal P(\xi)+\mathcal Q(\eta)+-\mathcal{Q}_1(\boldsymbol{\alpha})\right)\varphi_{-1}+\frac{\dot\mu}{\mu} \varphi_0\\
 +\left(\dot\alpha_{1}-\alpha_{1}+\frac{\dot\mu}{\mu}\alpha_{1}+\mathcal P (\xi)+\mathcal Q (\eta)\right)\varphi_{1}=0
\end{multline}
for $y\in \Omega^*(s)$. Since now all nonlinear terms on the  first line  of \eqref{eq-eta-inter-2} are $O(|y|^3)$ near the origin, we will choose the time evolution of the parameters $\mu$ and $\boldsymbol{\alpha}$ to cancel the projections on $\textup{Span}(\{\varphi_{i}\}_{i=-1,0,1})$ appearing on the second and third lines. This leads to the following modulation equations
\begin{align}
\label{modulation-mu} & \frac{\dot\mu}{\mu}= 0,\\
 \label{modulation-alpha-1} &\dot\alpha_{-1}=-\alpha_{-1}-\mathcal P(\xi)-\mathcal Q(\eta)+\mathcal Q_1(\boldsymbol{\alpha}), \\
  \label{modulation-alpha1} &\dot\alpha_{1}=\alpha_{1}-\mathcal P (\xi)-\mathcal Q (\eta).
\end{align}

Consequently,   equation \eqref{eq-eta-inter-2} becomes  
\begin{multline}\label{eq-xi}
 \partial_s \xi - L (\xi) + \left(\int_{0}^y \eta(s,  \bar y) \,d\bar y\right) \partial_y \xi -\left(2 \beta+\xi \right)\xi\\
 + \left(\int_{0}^y \eta (s, \bar y) \,d\bar y\right) \partial_y \beta- \mathcal{N}_1(s, y)=0,
\end{multline}
for $y\in\Omega^*.$
The system of the PDE \eqref{eq-xi} coupled with the four ODEs \eqref{eq-x*}, \eqref{modulation-mu}-\eqref{modulation-alpha-1}-\eqref{modulation-alpha1}  is supplemented by the choice of initial data
\begin{equation}\label{definition-xi0}
\left\{
\begin{aligned}
&x^*(0)=x^*_0, \qquad \mu(0)=\mu_0, \qquad \alpha_{-1}(0)=:\alpha_{-1,0}, \qquad  \alpha_{1}(0)=:\alpha_{1,0},\\
& \xi(0,y)=:\xi_0(y)=\frac{1}{\mu_0}a_0(x_0^*+y)-\cos(y)- \alpha_{-1,0}\varphi_{-1}(y)-\alpha_{1,0} \varphi_1(y),
\end{aligned}\right.
\end{equation}
where we used \eqref{definition-eta0} for the last identity.

\subsection{Solving the first two modulation equations} 

The values for the parameters $\mu$ and $\alpha_{1}$ can be obtained explicitly:

\begin{lem} \label{lem:mu-alpha1}

We have
\begin{align} \label{solution-modulation-mu}
& \mu(s)=\mu_0 ,\\
  \label{solution-modulation-alpha1}
& \alpha_1(s)=-\frac{1}{4\pi}\int_{-\pi}^\pi \xi(s,y)dy ,
\end{align}
for all times $s$. In particular, if $|x^*(s)|$ is small enough,
\begin{equation}  \label{solution-modulation-alpha1-bis}
| \alpha_1(s)|\lesssim \| \xi\|_{L^\infty(\Omega^*(s))}+|x^*||\alpha_{-1}(s)|
\end{equation}
\end{lem}

The fact that the parameter $\mu$ is time-independent is a consequence of the monotonicity and conserved quantities along characteristics starting at the maximum of the function, see Corollary \ref{cor:propagation-maxima}, and of $\partial_y^2 \varphi_{-1}(0)=0$ and $\partial_y^2 \varphi_{1}(0)=0,$ 
 see \eqref{taylor-origin-eigenfunctions}.

As for the expression \eqref{solution-modulation-alpha1} for $\alpha_1$, it is crucial for our analysis. Indeed, the first two terms in the modulation equation \eqref{modulation-alpha1} produce $\dot\alpha_{1}=\alpha_1$ which would predict instability by exponential growth. Therefore, without further information on the for\-cing terms $-\mathcal P (\xi)$ and $-\mathcal Q (\eta)$ in \eqref{modulation-alpha1}, one would believe that $\alpha_1$ grows exponentially fast. This is however not the case since $\alpha_{1}$ can be retrieved from $\alpha_{-1}$ and $\xi$ via the zero-mean condition \eqref{bc}.

\begin{proof}[Proof of Lemma \ref{lem:mu-alpha1}]

The identity \eqref{solution-modulation-mu} follows directly from \eqref{modulation-mu}. Then, one computes by \eqref{bc-renormalized-variables-2} and \eqref{mean-eigenfunctions}
$$
0=\frac{1}{2\pi} \int_{-\pi}^\pi \eta(y)dy=\frac{1}{2\pi} \int_{-\pi}^\pi \xi(y)dy+2\alpha_1
$$
and the desired identity follows. The estimate \eqref{solution-modulation-alpha1-bis} then follows from \eqref{bd:xi-L1--pi-pi}.
\end{proof}


 \subsection{Decomposition of the initial data}
 
 We choose $x_0^*$ to be the point at which $a_0$ attains its maximum. Then, in view of \eqref{modulation-equation-x*}, by Lemma \ref{lem:propagation-maximum} the point $x=x^*(t)$ 
  corresponding to $y=0$ is where $a(t,\cdot)$ attains its maximum. The point $x_0^*$ is uniquely determined by the following straightforward result from differential calculus:

\begin{lem} \label{lem:characterization-x0*}

For $\delta>0$ small enough, if $a_0(x)=\cos(x)+\tilde a_0(x)$ with $\| \tilde a_0\|_{C^2([-\pi,\pi])}\leq \delta$, then there exists a unique point $x^*_0$ at which $a_0$ attains its maximum, and it satisfies $|x^*_0|\leq C_0 \|   \tilde a'_0\|_{L^\infty([-\pi,\pi])}$,  for some universal positive constant $C_0.$

\end{lem}

Then, we choose $\mu_0$, $\alpha_{-1,0}$ and $\alpha_{1,0}$ to ensure that $\xi_0$ satisfies the vanishing conditions $\xi_0(0)= \xi'_0(0)=\xi''_0(0)=0$.

\begin{lem} \label{lem:initial-data}
For $\delta>0$ small enough, let $a_0(x)=\cos(x)+\tilde a_0(x)$ with $\| \tilde a_0\|_{C^3([-\pi,\pi])}\leq \delta$, and $x^*_0$ be given by Lemma \ref{lem:characterization-x0*}. Then the choice
$$
\mu_0=-\partial_x^2 a_0(x_0^*), \qquad \alpha_{-1,0}=1+\frac{a_0(x_0^*)}{\partial_x^2 a_0(x_0^*)}, \qquad \alpha_{1,0}=0
$$
ensures
$$
|\mu_0-1|\leq C_0 \delta \qquad {\rm and} \qquad |\alpha_{-1,0}|\leq C_0 \delta
$$
and that $\xi_0$ given by \eqref{definition-xi0} satisfies
\begin{equation} \label{taylor-origin-initial-xi}
\xi_0(0)=\partial_y \xi_0(0)=\partial_y^2 \xi_0(0)=0
\end{equation}
and
$$
\| \xi_0\|_{L^\infty([-\pi-x_0^*,\pi-x_0^*])}+\sum_{j=1}^3 \| \partial_y^j \eta_0 \|_{L^\infty([-\pi-x_0^*,\pi-x_0^*])}\leq C_0 \delta,
$$
for some universal positive constant $C_0.$
\end{lem}

We remark that the function $\eta_0$, considered as a function on $\mathbb R$ obtained by $2\pi$-periodic extension, may not be differentiable at $-\pi-x_0^*$ and $\pi-x_0^*$, but that it is indeed $C^3$ on $[-\pi-x_0^*,\pi-x_0^*]$ that corresponds to the assumption $a_0\in  C^3([-\pi, \pi])$.

\begin{proof}

This is a direct consequence of the identities \eqref{definition-xi0} and \eqref{taylor-origin-eigenfunctions} together with expansion for $a_0(x)$ at $x_0^*.$
\end{proof}


 \section{Nonlinear stability}\label{S5}

 
We are ready to prove Theorem \ref{Th-main},  that is the global-in-time nonlinear stability of the steady state $\cos x$ for \eqref{eq-main}.  To do that,   we need  at least a short time  local wellposedness  result for \eqref{eq-main} in H\"older spaces and  a regularity criterion in terms of $\|a\|_{L^1(0,  t;  L^\infty([-\pi, \pi ])}.$

\subsection{Local well-posedness}

We recall that the weight $W_\theta$ is defined by \eqref{id:def-Wtheta}.

\begin{prop}\label{LWP}
Let $C_0$ be as in Lemma \ref{lem:initial-data}. Suppose that \eqref{Th-icmain} is satisfied for $\delta>0$ small enough. Then the problem
\eqref{eq-main}  admits a unique local-in-time $C^3(\Omega)$ solution $a,$ such that it holds
 \begin{align}
  \left\|\frac{\xi(s, \cdot)}{W_\theta}\right\|_{L^\infty(\Omega^*)}\leq& 2C_0 \delta, \label{bd:bootstrap1}\\
    |x^*(s)|\leq& 2C_0 \delta, \label{bd:bootstrap2}\\
|\boldsymbol{\alpha}(t)|\leq& 2C_0\delta,\label{bd:bootstrap3}
 \end{align}
 on a short time interval $[0,t_0]$, where $(\xi,  x^*,   \boldsymbol\alpha)$ is the corresponding solution of the problem \eqref{eq-eta}-\eqref{definition-eta0} with initial data given by Lemmas \ref{lem:characterization-x0*} and \ref{lem:initial-data}.

Moreover,   the $ C^3$ regularity of any general solution $a$ in any closed interval $[0, T]$ is controllable by
\begin{equation}\label{BKM}
\|a(t,   \cdot)\|_{ C^3(\Omega)}\leq (\|a_0\|_{ C^3(\Omega)}+\|a_0\|_{ C^3(\Omega)}^3)(t+1)^2\,e^{C\int_0^t \|a(\bar t,  \cdot)\|_{L^\infty(\Omega)}\,d\bar t}\quad {\rm for}~t\in[0, T].
\end{equation}

 \end{prop}
 
 \begin{proof}
 The local-in-time well-posedness of Equation \eqref{eq-b} for any initial data $b_0\in H^2$ is proved in \cite{CCGO} via vanishing viscosity limit.    
  We remark that  Equation  \eqref{eq-main} is a transport type equation and to work out an existence   result   in the H\"older function spaces it is    convenient to use the characteristics method to avoid loss of derivatives.    Since the argument is quite classical,  the details are omitted. 
We would like to pay attention to the proof of the regularity criterion  \eqref{BKM}.  Define the characteristics associated to Equation \eqref{eq-main} by
 \begin{equation*} 
\frac{d X(t, z)}{dt}=\int_{-\pi}^{X(t, z)} a(t,  \bar z)\,d\bar z,\qquad X(t, z)|_{t=0}=z\in \Omega,
\end{equation*}
and  $\bar{a}(t, z)=a(t, X(t, z)),$  then the Jacobian $J(t, z)=\frac{d X}{dz}=\exp(\int_0^t\bar{a}(\bar t, z)\,d\bar{t})$ and
the equation of $(\partial_x a)(t, X(t, z)),~(\partial^2_x a)(t, X(t, z)),~(\partial^3_x a)(t, X(t, z))$ respectively read
\begin{align*}
& \frac{d  }{dt} (\partial_x a)(t, X(t, z))=\bar{a} ~(\partial_x a)(t, X(t, z)),\\
 & \frac{d  }{dt} (\partial^2_x a)(t, X(t, z))= \left((\partial_x a)(t, X(t, z))\right)^2,\\
  &\frac{d  }{dt} (\partial^3_x a)(t, X(t, z))=2  (\partial_x a)(t, X(t, z))~ (\partial^2_x a)(t, X(t, z))-\bar{a} ~(\partial^3_x a)(t, X(t, z)),
\end{align*}
In the case when $\int_0^t \|a(\bar t, \cdot)\|_{L^\infty(\Omega)}\,d\bar{t}$ is bounded on $[0, T],$ we find that
\begin{align*}
&\|\partial_x a(t, \cdot)\|_{L^\infty(\Omega)}=\|(\partial_x a)(t, X(t, \cdot))\|_{L^\infty(\Omega)}\leq \|\partial_x a_0\|_{L^\infty(\Omega)}\exp\left(\int_0^t \|a(\bar t, \cdot)\|_{L^\infty(\Omega)}\right),\\
&\|\partial^2_x a(t, \cdot)\|_{L^\infty(\Omega)}\leq  \|a_0\|_{ C^2(\Omega)} + \int_0^t \|(\partial_x a)(\bar t, X(t, \cdot))\|_{L^\infty(\Omega)}^2\,d\bar t,\\
&\|\partial^3_x a(t, \cdot)\|_{L^\infty(\Omega)}\leq  \left(\|a_0\|_{ C^3(\Omega)}  + 2\int_0^t \|\partial_x a (\bar t,  \cdot)\|_{L^\infty(\Omega)}  \,\|\partial^2_x a (\bar t,  \cdot)\|_{L^\infty(\Omega)}   \,d\bar t\right)\\
&\hspace*{4cm}\cdot\exp\left(\int_0^t \|a(\bar t, \cdot)\|_{L^\infty(\Omega)}\right).
\end{align*}
Hence
\begin{equation*}
\|a(t, \cdot)\|_{ C^3(\Omega)} \leq  (\|a_0\|_{ C^3(\Omega)}+\|a_0\|_{ C^3(\Omega)}^3) (t+1)^2\exp\left(C\int_0^t \|a(\bar t, \cdot)\|_{L^\infty(\Omega)}\right).\qedhere
\end{equation*}
 \end{proof}

 \subsection{Bootstrap argument}

At this stage,  we are going to apply the  linear decay estimates from Section \ref{sec3} and a bootstrap argument to show that the nonlinear stability holds for all time.
   We denote  $T^\infty$  to be the largest time such that the solution $a(t,  \cdot)\in C^3(\Omega)$ obtained from Proposition \ref{LWP} exists.  Hence to prove Theorem \ref{Th-main},  we only need to show that $T^\infty=\infty$ and that \eqref{a-behavior} holds.  Now let $C_1$ be a  positive constant,  which will be determined later on,  we define  
   \begin{multline}\label{def-T*}
   T^*=\sup \left\{s\in[0, T^\infty):
   \left\| \frac{\xi(s,  \cdot)}{W_\theta}\right\|_{L^\infty(\Omega^*(s))} + |x^*(s)|+|\boldsymbol{\alpha}(s)|\leq  C_1 \delta  ~e^{-(1-\theta')s}\right \}
   \end{multline}
    
    In what follows, we shall prove that ${T^*=T^\infty}=\infty$ under the assumption \eqref{Th-icmain}.   By contradiction,   let us suppose that ${T^*<\infty}.$   
   Keeping the notations of Proposition \ref{Prop-quasi},   it enables one to  write  the solution $\xi$ to \eqref{eq-xi} obtained from Proposition \ref{LWP} as 
    \begin{equation}\label{Dformula-xi}
    \xi= S(s, 0)\xi_0+ \int_0^s S(s,  \bar s) F(\bar s, y)\,d\bar s \quad{\rm for}~~s\in[0, T^*],
    \end{equation}
    where 
    \begin{equation*} 
F(s, y)= \left(2 \beta+\xi \right)\xi-
   \left(\int_{0}^y \eta(s, \bar y) \,d\bar y\right) \partial_y \beta+ \mathcal{N}_1(s, y).
\end{equation*}
   Applying  Proposition \ref{Prop-quasi} to \eqref{Dformula-xi} implies that 
    
   \begin{align}\label{es-xi}
 \left\|\frac{\xi(s, \cdot)}{W_\theta} \right\|_{L^\infty(\Omega^*(s))}\leq&  \left\|\frac{S(s, 0)\xi_0}{W_\theta}\right\|_{L^\infty(\Omega^*(0))}+ \int_0^s  \left\|\frac{S(s,  \bar s) F(\bar s,  \cdot)}{W_\theta} \right\|_{L^\infty(\Omega^*(s))}\,d\bar s\nonumber\\
   \leq&  e^{-(1-\theta')s}\left\|\frac{ \xi_0}{W_\theta}\right\|_{L^\infty(\Omega^*(0))} +\int_0^s e^{-(1-\theta')(s-\bar s)} \left\|\frac{ F(\bar s, \cdot)}{W_\theta}\right\|_{L^\infty(\Omega^*(\bar s))}\,d\bar s.
    \end{align}
We can estimate each term of  $F(\bar s, y)$ on $[0, s]$ as follows.
\begin{align*}
\left\|\frac{\left( 2\beta+\xi \right)\xi(\bar s, \cdot)}{W_\theta}\right\|_{L^\infty(\Omega^*(\bar s))}\lesssim \left(|\boldsymbol\alpha(s)| +\|\xi(s, \cdot)\|_{L^\infty(\Omega^*(\bar s))}\right)\left\|\frac{\xi(\bar s, \cdot)}{W_\theta}\right\|_{L^\infty(\Omega^*(\bar s))}
\end{align*}
    and in virtue of  \eqref{taylor-origin-eigenfunctions}
    \begin{align*}
  \left\|\frac{\left(\int_{0}^y \eta(\bar s, \bar y) \,d\bar y\right) \partial_y \beta}{W_\theta}\right\|_{L^\infty(\Omega^*(\bar s))}   &\lesssim  \|\eta(\bar s,  \cdot) \|_{L^\infty(\Omega^*(\bar s))} ~ \left\|\frac{y\partial_y \beta}{W_\theta}\right\|_{L^\infty(\Omega^*(\bar s))}\\
  &\lesssim |\boldsymbol\alpha(\bar s)|~ \|\eta(\bar s,  \cdot) \|_{L^\infty(\Omega^*(\bar s))},
    \end{align*}
    and of \eqref{taylorexpansionN}
    \begin{equation*}
      \left\|\frac{\mathcal{N}_1(\bar s,  \cdot)}{W_\theta}\right\|_{L^\infty(\Omega^*(\bar s))}\lesssim |\boldsymbol\alpha(\bar s)|^2
    \end{equation*}
    
 Substituting above estimates into \eqref{es-xi} and using of  \eqref{def-T*} gives that
   \begin{multline*} 
   e^{(1-\theta')s} \left\| \frac{\xi(s, \cdot)}{W_\theta}\right \|_{L^\infty(\Omega^*(s))}\leq 
 \left\|\frac{ \xi_0}{W_\theta}\right\|_{L^\infty(\Omega^*(0))}\\
 \qquad  \qquad \qquad+C\int_0^se^{(1-\theta') \bar s} \left( C_1\delta   e^{(1-\theta')\bar s} \left\| \frac{\xi(\bar s,  \cdot)}{W_\theta}\right\|_{L^\infty(\Omega^*(\bar s))}+   (C_1\delta)^2  e^{-2(1-\theta')\bar s} \right)\,d\bar s,
    \end{multline*}
     Gr\"onwall's lemma  then implies that 
     \begin{equation}\label{es-xi2}
 \left \| \frac{\xi(s, \cdot)}{W_\theta} \right\|_{L^\infty(\Omega^*(s))}\leq C e^{\frac{C_1\delta}{1-\theta'}}~\left(C_0 \delta+\frac{(C_1\delta)^2}{1-\theta'} \right)~
   e^{-(1-\theta')s}.
    \end{equation}
    
    It remains to estimate $x^*$ and $\boldsymbol\alpha.$  Notice that $x^*(s)$ is small on $[0, T^*],$  we write the  equation \eqref{eq-x*} of $x^*$ to
    \begin{equation*}
    \frac{d}{ds}\sin(x^*)=-\sin(x^*)+(1-\cos(x^*))\sin(x^*)+\cos(x^*)\int_{-\pi-x^*}^0 \eta(  s,  y)\,dy
    \end{equation*}
    and  find that
    \begin{equation*}
    | \sin(x^*(s))|\leq  \left(|\sin(x^*_0)|e^{-s}+  \int_0^s\int_{-\pi-x^*(\bar s)}^0 \eta( \bar s,  y)\,dy \, d\bar s     \right)~e^{\frac{(C_1\delta)^2}{1-\theta'}}.
    \end{equation*}
    This together with Lemma \ref{lem:characterization-x0*} gives that 
    \begin{equation}\label{es-x*}
|x^*(s)| \leq C\,e^{\frac{(C_1\delta)^2}{1-\theta'}}~ \left(C_0 \delta\,e^{-s}+  \int_0^s  \left( \left\| \frac{\xi(\bar s, \cdot)}{W_\theta} \right\|_{L^\infty(\Omega^*(\bar s))}+|\boldsymbol\alpha(\bar s)|\right)  \, d\bar s   \right)
    \end{equation}
  One can  estimate   $\boldsymbol\alpha$ similarly to get
   \begin{equation}\label{es-alpha}
| \boldsymbol\alpha  (s)| \leq C \,e^{\frac{C_1\delta}{1-\theta'}}~  \left(C_0 \delta \,e^{-s}+  \int_0^s  \left\| \frac{\xi(\bar s, \cdot)}{W_\theta} \right\|_{L^\infty(\Omega^*(\bar s))}  \, d\bar s   \right),
    \end{equation}
   the key point is that $\alpha_{-1}$  satisfies \eqref{modulation-alpha-1} and  $\alpha_1$ is given by \eqref{solution-modulation-alpha1} and satisfies \eqref{solution-modulation-alpha1-bis}.
      At this stage,  since $C, ~C_0, ~\theta, ~\theta'$ are just universal constants,    we can take  $\delta$   small enough  and  fix $C_1$ satisfying
    \begin{equation}
    CC_0 e^{\frac{C_1\delta}{1-\theta'}}<\frac{C_1}{8},\quad Ce^{\frac{C_1\delta}{1-\theta'}}{\frac{C_1\delta}{(1-\theta')^2}}<\frac{1}{8},
    \end{equation}
such that 
   \eqref{es-xi2}  becomes  
       \begin{equation}\label{es-xi3}
  \left\| \frac{\xi(s, \cdot)}{W_\theta}\right \|_{L^\infty(\Omega^*(s))}\leq \frac{C_1\delta}{4}\,e^{-(1-\theta')s}\quad {\rm for}~~s\in[0, T^*].
    \end{equation}
Consequently  \eqref{es-x*} and \eqref{es-alpha} can be rewritten to
       \begin{equation} 
 |x^*(s)| \leq  \frac{C_1\delta}{4}\,e^{-(1-\theta')s},
    \end{equation}
         \begin{equation} 
 |\boldsymbol\alpha(s)| \leq  \frac{C_1\delta}{4}\,e^{-(1-\theta')s},
    \end{equation}
    which contradicts with \eqref{def-T*}.   Hence we conclude that $T^*=\infty,$ and there hold   \eqref{bd:bootstrap1},   \eqref{bd:bootstrap2} and  \eqref{bd:bootstrap3} on $[0,  \infty)$.    Thanks to the regularity criterion \eqref{BKM},  we see that the solution $a(t,  \cdot)\in C^3(\Omega)$ exists globally in time.  We find that $a(t, x)=\mu_0(\cos(y)+\eta(s, y))$ with  $\mu_0$ given in   Lemma \ref{lem:initial-data},    one  gets from  \eqref{bd:bootstrap1},   \eqref{bd:bootstrap2} and  \eqref{bd:bootstrap3} that 
    \begin{equation}
\|a(t, \cdot)-\mu_0 \cos(\cdot)\|_{L^\infty(\Omega)}\lesssim \delta\, e^{-\mu_0(1-\theta')t},\quad \forall~t\geq0.
\end{equation}
    
    There remains to prove the convergence for the derivative in \eqref{a-behavior}. We consider the equation of $\partial_y \eta$ which reads
    \begin{multline}\label{eq-d-eta}
    \partial_s (\partial_y \eta) + \left( \sin (y)+\int_0^y\eta(s, \bar{y})\,d\bar{y}      \right) \partial_y (\partial_y \eta)-(\eta+\cos(y))\partial_y \eta\\=\cos(y)\int_0^y\eta(s, \bar{y})\,d\bar{y}-\sin(y)\,\eta.
    \end{multline}
We claim the following result, for which the proof is similar to the case of Lemma \ref{Le-quasi}.
\begin{lem}\label{lemma:quasilinear-derivative}
Let $\tilde S(s,s_0)u_0$ denote the solution to
$$
  \partial_s u + \left( \sin (y)+\int_0^y \eta (s, \bar{y})\,d\bar{y}      \right) \partial_y u -(\eta+\cos(y))u =0.
$$
Then, for the weight
$$
\omega =\sin^2(y/2)
$$
there holds
$$
\left\|\frac{\tilde{S}(s,s_0)u_0}{\omega}\right\|_{L^\infty{(\Omega^*(s))}}\leq  e^{-(1-\theta')(s-s_0)}~
 \left\|\frac{u_0}{\omega}\right\|_{L^\infty{(\Omega^*_{s_0})}}.
$$
\end{lem}
With this Lemma we are able to write the solution to \eqref{eq-d-eta} using Duhamel's formula as
$$
\partial_y \eta(s)= \tilde S(s,0)\eta_0 + \int_0^s \tilde S(s,\bar s)\left(\cos(y)\int_0^y\eta(\bar s, \bar{y})\,d\bar{y}-\sin(y)\,\eta(\bar s)\right)d\bar s
$$
and to estimate it by
\begin{align*}
& \| \frac{\partial_y \eta(s)}{\omega}\|_{L^\infty(\Omega^*(s))} \\ 
& \lesssim e^{-(1-\theta')s}~
 \left\|\frac{\partial_y \eta_0}{\omega}\right\|_{L^\infty{(\Omega_0^*)}} + \int_0^s e^{-(1-\theta')(s-\bar s)}\left\| \frac{\cos(y)\int_0^y\eta(\bar s, \bar{y})\,d\bar{y}-\sin(y)\,\eta(\bar s)}{\omega}\right\|_{L^\infty}d\bar s \\
 & \lesssim \delta e^{-(1-\theta')s}+\delta \int_0^s e^{-(1-\theta')(s-\bar s)} e^{-(1-\theta')\bar s}d\bar s \lesssim (1+\bar s)e^{-(1-\theta')s}
\end{align*}
where we used for the first term that $\partial_{yy} \eta_0(0)=0$ by \eqref{taylor-origin-eigenfunctions} and \eqref{taylor-origin-initial-xi} so that $\|\frac{\partial_y \eta_0}{\omega}\|_{L^\infty{(\Omega_0^*)}}<\infty$, and for the second term the bootstrap inequalities \eqref{def-T*}. This ends the proof of \eqref{a-behavior} and of Theorem \ref{Th-main}.

        \qed

 
   \section{Nonlinear instability} \label{sec:instability}

Let $\epsilon>0$. Then for any $\sigma>0$ we can find an initial data $a_0\in C^{2-\epsilon}(\Omega)$ that satisfies the following properties, which we will prove below will lead to an instability mechanism.  The function $a_0$ attains its maximum at the origin.  It is even with respect to $x$. It satisfies
$$
\| a_0-\cos (\cdot)\|_{C^{2-\epsilon}(\Omega)}\leq \sigma.
$$ 
Moreover, we have
\begin{equation} \label{initial-data-instability}
a_0(x)= \cos(x) +|x|^{2-\frac{\epsilon}{2}}
\end{equation}
for all $x$ that is close enough to $0$ (depending on $\epsilon$ and $\sigma$). Let us now assume by contradiction that the corresponding solution $a$ is global in time, and that there exists $\mu\in\R\setminus\{0\}$ such that
\begin{equation} \label{contradiciton-hp}
\|a(t, \cdot)-\mu \cos(\cdot)\|_{L^\infty(\Omega)}\to 0.
\end{equation}
In fact,  we can further assume that $\mu>0$ thanks to the fact that  $a_0$ attains its maximum at the origin and Lemma \ref{lem:propagation-maximum}.
We denote by $z(t)$ a general characteristic starting from a general point $z_0$. Since $a_0$ is even, then $a$ remains even over time and the mean free condition is equivalent to $\int_{-\pi}^0 a(t,x)dx=0$ for all $t$. In particular, the characteristics starting from the origin remains at the origin. The next Lemma studies quantitatively the stability of the origin by the characteristics flow.

\begin{lem} \label{lem:technical-instability}

Assume that \eqref{contradiciton-hp} holds true for some $\mu>0.$ Then, for all $\kappa_0>0$ and $\delta\in(0, \frac{1}{\mu})$,  there exists $\kappa\in(0,  \kappa_0)$ such that if initially $|z_0|\leq \kappa$ then
$$
|z(t)|\leq \kappa_0 \qquad \mbox{for all }0\leq t \leq t_0:= \left(\frac{1}{\mu}-\delta\right)\ln \left(\frac{\kappa_0}{|z_0|} \right).
$$

\end{lem}

\begin{proof}

Since $\int_{-\pi}^0 a(t,x)dx=0$ for all $t$, the velocity of a characteristics is $\dot z(t)=\int_0^z a(t,x)dx$ and satisfies $|\dot z|\leq |z|\| a(t,\cdot)\|_{L^\infty}$. By \eqref{contradiciton-hp} we have $ \| a(t',\cdot)\|_{L^\infty}\to \mu$ as $t\to \infty$ and we deduce that for all $\gamma>0$, there exists $C(\gamma)>0$ such that
$$
|z(t)|\leq |z_0|e^{\int_0^t \| a(t',\cdot)\|_{L^\infty}dt'}\leq |z_0|e^{C(\gamma)+t(\mu+\gamma)}.
$$
So if $0\leq t \leq  \left(\frac{1}{\mu}-\delta\right)\ln \left(\frac{\kappa_0}{|z_0|} \right)$ and $|z_0|\leq \kappa<\kappa_0$ we get from the above
\begin{align*}
|z(t)|\leq |z_0|e^{C(\gamma)}  \left(\frac{\kappa_0}{|z_0|} \right)^{(\frac{1}{\mu}-\delta)(\mu+\gamma)}=& \kappa_0 e^{C(\gamma)} \left(\frac{|z_0|}{\kappa_0}\right)^{1-(\frac{1}{\mu}-\delta)(\mu+\gamma)}\\
\leq&\kappa_0 e^{C(\gamma)} \left(\frac{\kappa}{\kappa_0}\right)^{1-(\frac{1}{\mu}-\delta)(\mu+\gamma)}
\leq \kappa_0,
\end{align*}
if one chooses first $\gamma$ small enough depending on $\delta,~\mu$, then $\kappa$ small enough depending on $\gamma$. This is the desired result.
\end{proof}

We can now end the proof of the instability theorem.

\begin{proof}[Proof of Theorem \ref{Th-main-2}]

We consider for $z$ a characteristics the difference $a(t,0)-a(t,z(t))$. By \eqref{eq-main} it solves
$$
\frac{d}{dt}\left(a(t,0)-a(t,z(t)) \right)=a^2(t,0)-a^2(t,z(t))=\left(a(t,0)+a(t,z(t)\right)(a(t, 0)-a(t,  z(t)).
$$
The solution to the above differential equation is
$$
a(t,0)-a(t,z(t))= e^{\int_0^t (a(s,0)+a(s,z(s))ds} (a_0(0)-a_0(z_0)).
$$
We now consider $\kappa_0$ and $\delta>0$ small and apply Lemma \ref{lem:technical-instability}, and we pick $z_0$ with $|z_0|\leq \kappa$ as in the statement of the lemma. Up to choosing a smaller $\kappa$, we can assume that \eqref{initial-data-instability} holds true for $z_0$ and the above becomes
$$
a(t,0)-a(t,z(t))= e^{\int_0^t (a(s,0)+a(s,z(s))ds} |z_0|^{2-\frac{\epsilon}{2}}.
$$
For times $t\leq t_0$ we have $|z(t)|\leq \kappa_0$, and we recall that $a(t, \cdot)\to \mu \cos(\cdot)$ in $L^\infty$ as $t\to \infty$ by the contradiction assumption \eqref{contradiciton-hp}. Combining these two informations, for any $\gamma>0$ we see that there exists a constant $C(\gamma)>0$ such that for all $0\leq t\leq t_0$ we have
$$
\int_0^t \left(a(s,0)+a(s,z(s))\right)ds\geq (2\mu-\gamma)t  -C(\gamma).
$$
Thus, $e^{\int_0^t (a(s,0)+a(s,z(s)))ds}\geq e^{-C(\gamma)} e^{(2\mu-\gamma)t}$, so that at time $t_0$,
\begin{align}
\nonumber \left|a(t_0,0)-a(t_0,z(t_0))\right| & \geq e^{-C(\gamma)} e^{(2\mu-\gamma)t_0} |z_0|^{2-\frac{\epsilon}{2}} \\
\nonumber & = e^{-C(\gamma)} \left(\frac{\kappa_0}{|z_0|} \right)^{(\frac{1}{\mu}-\delta)(2\mu-\gamma)} |z_0|^{2-\frac{\epsilon}{2}} \\
\label{id:instability:main-computation}& \geq e^{-C(\gamma)} \kappa_0^{(\frac{1}{\mu}-\delta)(2\mu-\gamma)}  |z_0|^{-\frac{\epsilon}{4}}
\end{align}
provided $\kappa$, $\delta$ and $\gamma$ have been chosen small enough depending on $\epsilon$.  We now fix all the constants and consider the above inequality as $z_0\to 0$. Since $t_0(z_0)\to \infty$ as $z_0\to 0$, the above implies that $\| a(t,\cdot)\|_{L^\infty}\to \infty$ as $t\to \infty$. But this contradicts \eqref{contradiciton-hp}, hence the result of the Theorem.
\end{proof}


\section{Proof of the stability and instability results for (IPM)} \label{sec:IPM-results}

This section is devoted to the proof of Theorems \ref{Th-IPM} and \ref{Th-main-4}. In order to prove the stability Theorem \ref{Th-IPM}, we will rely on the result of Theorem \ref{Th-main} concerning solutions of the reduced equation from the inviscid Primitive Equations proved in the previous sections. To prove the instability Theorem \ref{Th-main-4}, we will perform a minor adaptation of the proof of the corresponding instability Theorem \ref{Th-main-2} for the Primitive Equations.

\begin{proof}[Proof of Theorem \ref{Th-IPM}]

Assume $b_0$ is mean free and satisfies \eqref{Th-ic-IPM}. For $b$ the associated $C^3$ solution to \eqref{eq-b}, we perform the change of variables \eqref{def-a}, and get that $a$ is a solution of \eqref{eq-main}. We choose $\nu_0=\mu$, and thus the initial datum satisfies
$$
 \|a_0(\cdot) - \cos ( \cdot)\|_{C^3([-\pi,\pi])} \leq \delta.
$$
We apply Theorem \ref{Th-main}, and obtain that there exists $\mu_*=1+O(\delta)$ such that
\begin{equation} \label{asymptotic-a}
a(t,x)=\mu_* \cos(x)+O_{C^3}(\delta e^{- t/2}).
\end{equation}
We now unwind the change of variables. We abuse notations and write $\nu(t)=\nu(\tau (t))$ for the function $\nu$ expressed as a function of $t$. Using $\frac{dt}{d\tau}=\nu$, then \eqref{def-a}, and then changing variables $\bar \tau\mapsto t'$ in the integral, we have
\begin{align*}
\frac{d}{dt}\nu &= \frac{1}{\nu}  \frac{d}{d\tau}\nu = \frac{1}{\nu}\left(\frac{\nu}{\pi}\int_0^{\tau} \int_{-\pi}^{\pi}b^2(s,  x)\,dx\,d\bar\tau\right) \\
&=  \frac{1}{\pi}\int_0^{\tau} \int_{-\pi}^{\pi} \nu^2(t'(\bar \tau))a^2(t'(\bar \tau),  x)\,dx\,d\bar\tau \\
& = \frac{1}{\pi}\int_0^{t} \nu(t') \int_{-\pi}^{\pi}a^2(s,  x)\,dx dt' .
\end{align*}
This implies that $\nu$ solves the linear equation
$$
\partial_t^2 \nu = \Big(  \frac{1}{\pi} \int_{-\pi}^{\pi}a^2(s,  x)\,dx \Big)\nu = (\mu_{*}^2+O(\delta e^{-t/2}))\nu
$$
with initial datum $\nu(0)=\mu$ and $\partial_t \nu(0)=0$. We claim that
\begin{equation} \label{asymptotic-nu}
\nu(t)= \nu^* e^{\mu_* t}+\frac{\mu}{2} e^{-\mu_* t}+O(\mu \delta e^{(\mu_*-\frac 12)t}) = \nu^* e^{\mu_* t} (1+O(e^{-t/2}))
\end{equation}
for a constant $\nu^*=\frac{\mu}{2}(1+O(\delta))$, and postpone the proof of the claim to the end of this proof. Then using $\frac{d\tau}{dt}=\frac{1}{\nu}$ we have
$$
\tau = \frac{1}{\nu^*} \int_0^t \left( e^{-\mu_* t}+O(e^{-\mu_*-\frac 12 t})\right) dt .
$$
We deduce that there exists $\tau^*$ such that $\lim_{t\to \infty}\tau(t)=\tau^*$, and that $\tau^*-\tau= \frac{e^{-\mu_* t}}{\mu_*\nu^*}(1+O(e^{-t/2}))$ which gives
$$
e^{-\mu_* t}=\mu_*\nu^* (\tau^*-\tau)\left( 1+ O\left(|\mu (\tau^*-\tau)|^{\frac{1}{2\mu_*}} \right) \right).
$$
In turn, injecting the above in \eqref{asymptotic-a} and \eqref{asymptotic-nu} shows
$$
a=\mu_* \cos(x)+O_{C^3}\left(|\mu (\tau^*-\tau)|^{\frac{1}{2\mu_*}} \right)  ,  \qquad \nu =\frac{1}{\mu_* (\tau^*-\tau)}\left( 1+ O\left(|\mu (\tau^*-\tau)|^{\frac{1}{2\mu_*}} \right) \right).
$$
Using $b=\nu a$ and $\frac{1}{2\mu_*}\geq \frac 14$ since $\mu_*=1+O(\delta)$ this shows
$$
b(\tau,x)= \frac{\cos x}{\tau^*-\tau}+O_{C^3}\left( (\tau^{*}-\tau)^{-\frac 34} \right)
$$
Thus $b$ blows up in finite time $\tau^*$ with the desired asymptotics of Theorem \ref{Th-IPM}, ending the proof of the Theorem.

There remains to prove the claim \eqref{asymptotic-nu}. Let $\vec \nu=(\nu,\partial_t \nu)$, and $\vec w=e^{-\mu_* t}\vec \nu$. Then $\vec w$ solves
$$
\partial_t \vec w = \begin{pmatrix} -\mu_* & 1 \\ \mu_{*}^2 & -\mu_* \end{pmatrix} \vec w + O(\delta e^{-t/2}|\vec w|).
$$
We decompose $\vec w(t) = u_1(t) (1,\mu_*)+u_2(t)(1,-\mu_*)$. Then $(u_1,u_2)$ solves the system
$$
\left\{
\begin{array}{l l}
\partial_t u_1 = O(\delta e^{-t/2}(|u_1|+|u_2|), \\
\partial_t u_2 =-2\mu_* u_2+ O(\delta e^{-t/2}(|u_1|+|u_2|),
\end{array}
\right.
$$
with initial datum $(u_{1}(0),u_2(0))=(\frac{\mu}{2},\frac{\mu}{2})$. One can easily see from the above equation that there exists $u_{1}^\infty =\frac{\mu}{2}(1+O(\delta))$ such that
$$
u_1(t)= u_1^\infty +O(\mu \delta e^{-t/2}) \quad \mbox{and} \quad u_2(t)=\frac{\mu}{2} e^{-2\mu_* t}+O(\delta \mu e^{-t/2}).
$$
Using $\nu = e^{\mu_* t} (u_1+u_2)$ proves the desired claim \eqref{asymptotic-nu}, which ends the proof of Theorem \ref{Th-IPM}.
\end{proof}

\bigbreak
\bigbreak

\begin{proof}[Proof of Theorem \ref{Th-main-4}]

We pick $b_0=\mu \,a_0$ with $a_0$ the exact same initial datum used in Section \ref{sec:instability} to prove the instability Theorem \ref{Th-main-2} for the Primitive Equations. We suppose by contradiction that the result of Theorem \ref{Th-main-4} does not hold true, i.e. that the corresponding solution $b$ to \eqref{eq-main} blows up in finite time $\tau^*>0$ and satisfies \eqref{TH4-behavior}-\eqref{TH4-behavior-2}.

We change variables and define
$$
a(t,x)= \frac{\tau^*-\tau}{\mu \tau^*} b(\tau,x), \qquad t=\mu\tau^* \ln \left(\frac{\tau^*}{\tau^*-\tau}\right).
$$
Then $a$ solves
\begin{equation} \label{eq:modified-IPE-instability}
\partial_t a+  \left(\int_{-\pi}^x a(t, \bar x) \,d\bar x\right) \partial_x a -a^2 +\frac{1}{\pi}  \int_{-\pi}^{\pi}a^
2\,dx= \theta(t)a
\end{equation}
with initial datum $a(0,x)=a_0(x)$, where
$$
\theta(t)= -\frac{1}{\mu\tau^*}+\frac{\tau^*-\tau}{\mu\tau^*}\int_0^\tau \frac{1}{\pi} \int_{-\pi}^\pi b^2(\bar \tau,x)dxd\bar \tau \to 0
$$
as $t\to \infty$, because $ \frac{1}{\pi} \int_{-\pi}^\pi b^2(\bar \tau,x)dx = \frac{1+o_{\bar\tau\to \tau^*}(1)}{(\tau^*-\bar\tau)^2}$ from \eqref{TH4-behavior}-\eqref{TH4-behavior-2}. Moreover we have since $b$ satisfies \eqref{TH4-behavior}-\eqref{TH4-behavior-2} that $a$ satisfies $\|a(t, \cdot)-\bar\mu \cos(\cdot)\|_{L^\infty(\Omega)}\to 0$ as $t\to \infty$ for $\bar\mu=\frac{1}{\mu\tau^*}$.

The proof of Theorem \ref{Th-main-2} performed in Section \ref{sec:instability} adapts to the present case \eqref{eq:modified-IPE-instability} to show a contradiction, with only a minor modification that we now explain. We shall keep the notations of the proof of Theorem \ref{Th-main-2}. First, since the characteristics of Equations \eqref{eq-main} and \eqref{eq:modified-IPE-instability} are exactly the same, the result of Lemma \ref{lem:technical-instability} holds true in the present case. Then, we study the difference $a(t,0)-a(t,z(t))$, which satisfies by \eqref{eq:modified-IPE-instability}
$$
\frac{d}{dt}(a(t,0)-a(t,z(t)))= (a(t,0)+a(t,z(t))+\theta(t))(a(t,0)-a(t,z(t))
$$
so that $a(t,0)-a(t,z(t))=e^{\int_0^t (a(t',0)+a(t',z(t'))+\theta(t'))dt'}(a(0,0)-a(0,z_0)$. Since $\theta(t) \to 0$ as $t\to \infty$, we can argue as in the proof of Theorem \ref{Th-main-2} and obtain that for any $\gamma>0$, there exists $C(\gamma)>0$ such that for $z_0$ small enough and $0\leq t\leq t_0$ we have
$$
\int_0^t \left(a(s,0)+a(s,z(s))+\theta(s)\right)ds\geq (2\bar\mu-\gamma)t  -C(\gamma).
$$
Hence we obtain the same conclusion \eqref{id:instability:main-computation} in the present case too, namely 

$$
 \left|a(t_0,0)-a(t_0,z(t_0))\right|  \geq e^{-C(\gamma)} \kappa_0^{(\frac{1}{\bar\mu}-\delta)(2\bar\mu-\gamma)}  |z_0|^{-\frac{\epsilon}{4}}.
$$
Since $t_0(z_0)\to \infty$ as $z_0\to 0$, the above implies that $\| a(t,\cdot)\|_{L^\infty}\to \infty$ as $t\to \infty$. But this contradicts $\|a(t, \cdot)-\bar\mu \cos(\cdot)\|_{L^\infty(\Omega)}\to 0$, hence the result of the Theorem.
\end{proof}

 
\section*{Acknowledgement}

This result is part of the ERC starting grant project FloWAS that has received funding from the European Research Council (ERC) under the Horizon Europe research and innovation program (Grant agreement No. 101117820). C. Collot was supported by the CY Initiative of Excellence Grant   ``Investissements d'Avenir"   ANR-16-IDEX-0008 via Labex MME-DII and the grant ``Chaire Professeur Junior" ANR-22-CPJ2-0018-01 of the French National Research Agency. C. Collot and C. Prange were partially supported by the grant BOURGEONS ANR-23-CE40-0014-01 of the French National Research Agency.  C. Prange is also partially supported by project CRISIS grant ANR-20-CE40-0020-01 of the French National Research Agency and by the CY Initiative of Excellence project CYNA (CY Nonlinear Analysis). J. Tan was  supported by the CY Initiative of Excellence,  project CYNA (CY Nonlinear Analysis),  the Direct Grant (No.  4053715) of CUHK  and   the Hong Kong RGC grant CUHK14302525.  

\bigbreak

{\bf Data Availibility.}
Data sharing is not applicable to this article as no data sets were generated or analyzed during the current study.

\bigbreak

\end{document}